\documentclass[12pt, reqno]{amsart}
\usepackage{amsmath, amsthm, amscd, amsfonts, amssymb, graphicx, color}
\usepackage[bookmarksnumbered, colorlinks, plainpages]{hyperref}
\hypersetup{colorlinks=true,linkcolor=red, anchorcolor=green, citecolor=cyan, urlcolor=red, filecolor=magenta, pdftoolbar=true}
\input{mathrsfs.sty}
\usepackage{mathabx}

\newtheorem{theorem}{Theorem}[section]
\newtheorem{lemma}[theorem]{Lemma}
\newtheorem{proposition}[theorem]{Proposition}
\newtheorem{corollary}[theorem]{Corollary}
\theoremstyle{definition}

\newtheorem{example}[theorem]{Example}

\newtheorem{remark}[theorem]{Remark}
\numberwithin{equation}{section}

\begin{document}% Do not forget this command!
\title{$B$-spline interpolation problem in Hilbert $C^*$-modules}
\author[R. Eskandari, M. Frank, V. M. Manuilov, M. S. Moslehian]{Rasoul Eskandari$^1$, Michael Frank$^2$, Vladimir M. Manuilov$^3$,  \MakeLowercase{and} Mohammad Sal Moslehian$^4$}

\address{$^1$Department of Mathematics, Faculty of Science, Farhangian University, Tehran, Iran.}
\email{eskandarirasoul@yahoo.com}

\address{$^2$Hochschule f\"ur Technik, Wirtschaft und Kultur (HTWK) Leipzig, Fakult\"at Informatik und Medien, PF 301166, D-04251 Leipzig, Germany.}
\email{michael.frank@htwk-leipzig.de}

\address{$^3$ Moscow Center for Fundamental and Applied Mathematics, and Department of Mechanics and Mathematics, Moscow State University, Moscow, 119991, Russia.} 
\email{manuilov@mech.math.msu.su}

\address{$^4$Department of Pure Mathematics, Center of Excellence in Analysis on Algebraic Structures (CEAAS), Ferdowsi University of Mashhad, P. O. Box 1159, Mashhad 91775, Iran.}
\email{moslehian@um.ac.ir; moslehian@yahoo.com}



\begin{abstract}
We introduce the $B$-spline interpolation problem corresponding to a $C^*$-valued sesquilinear form on a Hilbert $C^*$-module and study its basic properties as well as the uniqueness of solution. We first study the problem in the case when the Hilbert $C^*$-module is self-dual. Extending a bounded $C^*$-valued sesquilinear form on a Hilbert $C^*$-module to a sesquilinear form on its second dual, we then provide some necessary and sufficient conditions for the $B$-spline interpolation problem to have a solution. 
\par Passing to the setting of Hilbert $W^*$-modules, we present our main result by characterizing when the spline interpolation problem for the extended $C^*$-valued sesquilinear to the dual $\mathscr{X}'$ of the Hilbert $W^*$-module $\mathscr{X}$ has a solution. As a consequence, we give a sufficient condition that for an orthogonally complemented submodule of a self-dual Hilbert $W^*$-module $\mathscr{X}$ is orthogonally complemented with respect to another $C^*$-inner product on $\mathscr{X}$.  Finally, solutions of the $B$-spline interpolation problem for Hilbert $C^*$-modules over $C^*$-ideals of $W^*$-algebras are extensively discussed. Several examples are provided to illustrate the existence or lack of a solution for the problem.
\end{abstract}
\maketitle

%---------------------------------------------------------------------------------------%
\section{Introduction}

The notion of \emph{spline} was introduced by Schoenberg \cite{SCH} and then, it has been applied in approximation theory, statistics, and numerical analysis, see \cite{ARC}. Later, this concept was generalized by several mathematicians as those elements in a Hilbert space $\mathscr{H}$ that minimize a specific bilinear form over translates of a certain null space. 

One of the most significant generalization of splines is due to Lucas \cite{LUC} in which he gives a formalization of the notion of $B$-spline given by others in the last years.  More precisely,  let $(\mathscr{H}, \langle\cdot,\cdot\rangle)$ be a real Hilbert space, let $\Lambda$ be a family of continuous linear forms over $\mathscr{H}$, and let $B(x, y)$ be a bounded bilinear form on $\mathscr{H} \times \mathscr{H}$ such that $B(x,x) \geq 0$ for all $x\in N(\Lambda) = \{x\in \mathscr{H}: \lambda(x) = 0 {\rm ~for ~all~} \lambda\in\Lambda\}$. 

A vector $s \in \mathscr{H}$ is called a \emph{$B$-spline} if $B(s, y) = 0$ for all $y \in N(\Lambda)$. The closed linear space of all $B$-splines is denoted by $Sp(B, \Lambda)$. For $x\in  \mathscr{H}$, an element $s \in \mathscr{H}$ is said to be a \emph{$\Lambda$-interpolate} of $x$ if $(s-x) \in N(\Lambda)$. If $s$ is also in $Sp(B, \Lambda)$, then $s$ is called an \emph{$Sp(B,\Lambda)$-interpolate} of $x$.
Lucas \cite{LUC} gives conditions that insure the existence of an $Sp(B, \Lambda)$-interpolate of any element in $\mathscr{H}$. One of the conditions is that the system $(\mathscr{H}, \Lambda, B, N(\Lambda))$ is well-posed in the sense that if $N_1:=\{x\in N(\Lambda): B(x,x)=0\}$, then $B(x,y)=0$ for all $x\in\mathscr{H}$ and all $y\in N_1$; moreover, if $N_2$  is the orthogonal complement of $N_1$ in $N(\Lambda)$, then $B$ is definite on $N_2$, that is, $B(x,x)\geq c\|x\|^2$ for all $x\in N_2$; see also \cite{ABR} for some existence and uniqueness conditions. 

%============================================%

Inspired by the theory of $B$-splines in Hilbert space setting, we investigate the $B$-spline interpolation problem in the framework of Hilbert modules over $C^*$-algebras and $W^*$-algebras. 

The paper is organized as follows. In the next section, we review some preliminaries required throughout the paper.

In Section 3, we introduce the $B$-spline interpolation problem corresponding to a $C^*$-valued sesquilinear form on a Hilbert $C^*$-module and give several examples to show the existence or the lack of a solution for the problem. We then give some characterizations under some mild conditions for the $B$-spline interpolation problem to have a solution and discuss the uniqueness of solution. Our characterizations are given when the closed subspace $\mathscr{Y}$ is orthogonally complemented in $\mathscr{X}$.  

Section 4 is devoted to study of the $B$-spline interpolation problem in the framework of Hilbert $W^*$-modules. We give necessary and sufficient conditions for the spline interpolation problem for the extended $C^*$-valued sesquilinear to the first dual of the Hilbert $W^*$-module to have a solution. As a consequence, we give a sufficient condition for that an orthogonally complemented submodule of a self-dual Hilbert $W^*$-module $\mathscr{X}$ is orthogonally complemented with respect to another $C^*$-inner product on $\mathscr{X}$. A concrete example is given to show that the conditions of the main result of the paper simultaneously occur. Another technical example shows that a crucial condition should be taken into account.

In Section 5 we deal with solutions of the $B$-spline interpolation problem for Hilbert $C^*$-modules over $C^*$-ideals of $W^*$-algebras. Several situations are discussed to give a comprehensive account of the problem in this setting. 

%============================================%

\section{Preliminaries}

The notion of a pre-Hilbert $C^*$-module is a natural generalization of that of an inner product space in which we allow the inner product to take its values in a $C^*$-algebra instead of the field of complex numbers. More precisely, a \emph{pre-Hilbert $C^*$-module} over a $C^{*}$-algebra $\mathscr A$ is a complex linear space $\mathscr{X}$ which is a right $\mathscr A$-module equipped with an $\mathscr A$-valued inner product $\langle\cdot,\cdot\rangle \,: \mathscr{X}\times \mathscr{X}\longrightarrow\mathscr A$ satisfying\\
(i) $\langle x, y + \lambda z\rangle =\langle x, y\rangle + \lambda\langle x, z\rangle$,\\
(ii) $\langle x, ya\rangle=\langle x, y\rangle a$,\\
(iii) $\langle x, y\rangle^*=\langle y, x\rangle$,\\
(iv) $\langle x, x\rangle\geq0$ and $\langle x, x\rangle=0$ if and only if $x=0$,\\
for all $x, y, z\in \mathscr{X}, a\in\mathscr A, \lambda\in\mathbb{C}$. As easy to see, the setting $\|x\|=\|\langle x, x\rangle\|^{\frac{1}{2}}$ defines a norm on $\mathscr{X}$. If $\mathscr{X}$ together with this norm is complete, then it is called a \emph{Hilbert $C^*$-module} over $\mathscr A$. The positive square root of $\langle x, x\rangle$ is denoted by $|x|$ for $x\in \mathscr{X}$. We say that a closed submodule $\mathscr{Y}$ of a Hilbert $C^*$-module $\mathscr{X}$ is orthogonally complemented if $\mathscr{X}=\mathscr{Y}\oplus \mathscr{Y}^{\perp}$, where $\mathscr{Y}^{\perp}=\{ x\in \mathscr{X} : \langle x,y\rangle=0$ for all $ y\in \mathscr{Y}\}$. 

Although Hilbert $C^*$-modules seem to be a natural generalization of Hilbert spaces, some of their basic properties are no longer valid in the setting of Hilbert $C^*$-modules in their full generality. For example, not any closed submodule of a Hilbert $C^*$-module is complemented and, not every bounded $C^*$-linear map on a Hilbert $C^*$-module is adjointable. 

If $\mathscr{X}'$ denotes the set of bounded $\mathscr{A}$-linear maps from $\mathscr{X}$ into $\mathscr{A}$, named as the dual of $\mathscr{X}$, then $\mathscr{X}'$ becomes a right $\mathscr{A}$-module equipped with the following actions:
\begin{eqnarray}\label{msm44}
(\rho+\lambda\tau)(x)=\rho(x)+\overline{\lambda}\tau(x)\quad {\rm~and~}\quad (\tau b)(x) =b^*\tau (x)
\end{eqnarray}
for $\tau \in \mathscr{X}', b\in \mathscr{A}, \lambda\in\mathbb{C}$ and $x\in \mathscr{X}$. Trivially, to every bounded $\mathscr{A}$-linear map $T: \mathscr{X} \to \mathscr{Y}$ one can associate a bounded $\mathscr{A}$-linear map $T':  \mathscr{Y}' \to  \mathscr{X}'$ defined by $T'(g)(x)=g(T(x))$ for $g\in  \mathscr{Y}'$.

For each $x\in \mathscr{X}$, one can define the map $\widehat{x}\in \mathscr{X}'$ by $\widehat{x}(y)=\langle x,y\rangle$ for $y\in \mathscr{X}$. It is easy to verify that the map $x \mapsto \widehat{x}$ is isometric and ${\mathscr{A}}$-linear. Hence one can identify ${\mathscr{X}}$ with $\widehat{\mathscr{X}}:=\{\widehat{x}: x\in \mathscr{X}\}$ as a closed submodule of $\mathscr{X}'$. 

A module $\mathscr{X}$ is called \emph{self-dual} if $\widehat{\mathscr{X}}=\mathscr{X}'$. For example, a unital $C^*$-algebra $\mathscr{A}$ is self-dual as a Hilbert $\mathscr{A}$-module via $\langle a,b\rangle=a^*b$. 

Given $x\in \mathscr{X}$, one can define $\dot{x}\in \mathscr{X}''$ by $\dot{x}(f)=f(x)^*\,\, (f\in\mathscr{X}')$. Then $x\mapsto \dot{x}$ gives rise to an isometric $\mathscr{A}$-linear map. We say that $\mathscr{X}$ is \emph{reflexive} if this map is surjective. It is known \cite[Chapter 4]{MT} that there is an $\mathscr{A}$-valued inner product on the second dual $\mathscr{X}''$ defined by $\langle F, G\rangle=F(\dot{G})$, where $\dot{G}(x):=G(\widehat{x})\,\,(x\in \mathscr{X})$. It is an extension of the inner product on $ \mathscr{X}$. In addition, the map $F\mapsto \dot{F}$ is an isometric inclusion, and $\dot{\dot{x}}=\widehat{x}$ because of
\begin{align*}
\dot{\dot{x}}(y)=\dot{x}(\widehat{y})=\widehat{y}(x)^*=\langle x,y\rangle=\widehat{x}(y).
\end{align*}
Thus we have the chain of inclusions as $ \mathscr{X}\subseteq  \mathscr{X}''\subseteq  \mathscr{X}'$, and every self-dual Hilbert $C^*$-module is reflexive, too.

If $\mathscr{A}$ is a $W^*$-algebra with the predual $\mathscr{A}_*$ consisting of all normal bounded linear functionals, then the $\mathscr{A}$-valued inner product $\langle\cdot,\cdot\rangle$ on $\mathscr{X}$ can be extended to an $\mathscr{A}$-valued inner product on $\mathscr{X}'$. We frequently use the construction given by \cite{Pa}: For a positive linear functional $f$ on $\mathscr{A}$, one can consider the semi-inner product $f(\langle \cdot, \cdot\rangle)$ on $\mathscr{X}$. Let $\mathscr{N}_f = \{x:f(\langle x,x\rangle)=0\}$. Then the quotient space $\mathscr{X}/\mathscr{N}_f$ equipped with the following inner product 
\begin{eqnarray}\label{Pas0}
(x + \mathscr{N}_f, y + \mathscr{N}_f)_f := f(\langle y,x\rangle )\quad (x, y \in \mathscr{X})
\end{eqnarray}
provides a pre-Hilbert space. Let us denote the Hilbert space completion
of $\mathscr{X}/\mathscr{N}_f$ by $\mathscr{H}_f$ and write $\|\cdot\|_f$ for the norm on $\mathscr{H}_f$ obtained from its inner product $(\cdot,\cdot)_f$. For $\tau\in \mathscr{X}'$, the map $x+\mathscr{N}_f \mapsto f(\tau(x))$ is a bounded linear functional with norm less than or equal to $\|\tau\|\,\|f\|^{1/2}$. From the Riesz representation theorem we conclude that there is a unique vector $\tau_f\in \mathscr{H}_f$ such that $\|\tau_f\|_f\leq \|\tau\|\|f\|^{1/2}$ and 
\begin{eqnarray}\label{Pas1}
(x+\mathscr{N}_f,\tau_f)_f=f(\tau (x))
\end{eqnarray}
for all $x\in \mathscr{X}$. Further construction of this type can be found in \cite{ARA, ESK,  FRA2}. The comprehensive result of Paschke reads as follows.

\begin{theorem}\cite[Theorem 3.2]{Pa}\label{th1}
Let $\mathscr{X}$ be a pre-Hilbert $C^*$-module over a $W^*$-algebra $\mathscr{A}$. The $\mathscr{A}$-valued inner product $\langle \cdot,\cdot\rangle$ can be extended to $\mathscr{X}'\times \mathscr{X}'$ in such a way as to make $\mathscr{X}'$ into a self-dual Hilbert $\mathscr{A}$-module. In particular, the extended inner product satisfies \begin{eqnarray}\label{Pas2}
\langle \tau,\widehat{x}\rangle=\tau (x)
\end{eqnarray}
and 
\begin{eqnarray}\label{Pas3}
f(\langle\tau,\rho\rangle) =(\tau_f,\rho_f)_f
\end{eqnarray}
for all $ x\in \mathscr{X}, \tau, \rho\in \mathscr{X}'$, and all normal positive linear functionals $f$ on $\mathscr{A}$.
\end{theorem}

By an \emph{$\mathscr{A}$-sesquilinear} (bounded) form on a Hilbert $\mathscr{A}$-module $\mathscr{X}$ we mean a map $B: \mathscr{X} \times\mathscr{X}\to \mathscr{A}$ such that it is anti-$\mathscr{A}$-linear in the first variable and $\mathscr{A}$-linear in the second one.  We say that it is \emph{positive} on a set $\mathscr{Y}$ if $B(y,y)\geq0$ for all $y \in \mathscr{Y}$.

 It is \emph{elliptic} on a set $\mathscr{Y}$ if $B(y,y)\geq c\langle y, y\rangle$ for any $y\in \mathscr{Y}$ and some positive real constant $c$. For Hilbert $C^*$-modules we consider another ellipticity condition: $B$ is \emph{coercive} on $\mathscr{Y}$ if $B$ is positive on $\mathscr{Y}$ and if there exist $c,k>0$ such that for any pure state $f$ on $\mathscr A$ and any $x\in \mathscr{Y}$ there exists a unit vector $y\in \mathscr{Y}$ such that $f(|y|^2)\geq k$ and $|f(B(x,y))|^2\geq cf(|x|^2)f(|y|^2)$. For Hilbert spaces these two conditions are equivalent. Indeed, the second condition means that there exists $c>0$ such that for any $x\in \mathscr{Y}$ there exists $y\in \mathscr{Y}$ such that $|B(x,y)|\geq c\|x\|\|y\|$. Let $T$ satisfy $B(x,y)=\langle Tx,y\rangle$, and suppose that $T$ is not bounded from below. Then there exists a sequence of unit vectors $\{x_n\}$ such that $Tx_n\to 0$. Let $y_n$ satisfy $\|y_n\|=1$ and $|\langle Tx_n,y_n\rangle|\geq c\|x_n\|\|y_n\|$, which gives a contradiction. Hence, $T$ is bounded from below. It is also positive, hence $|B(x,x)|\geq c\|x\|^2$ for any $x\in \mathscr{Y}$. In the opposite direction, if $|B(x,x)|\geq c\|x\|^2$ for any $x\in \mathscr{Y}$ then one can take $y=x$ to satisfy the second condition. 

Throughout the paper, let $\mathscr{A}$ be a $C^*$-algebra ($W^*$-algebra if we explicitly state it) whose pure state space is denoted by $ \mathcal P\mathcal S(\mathscr{A})$, and let $\mathscr{X}$ denote a Hilbert $\mathscr{A}$-module. The reader is referred to \cite{R} for terminology and notation on $C^*$-algebra and to \cite{LAN, MT, FRA1} for some basic knowledge on the theory of Hilbert $C^*$-modules.

%=================================================%

\section{$B$-splines in Hilbert $C^*$-modules}

Let $\mathscr{Y}$ be a closed submodule of $\mathscr{X}$. Let $B:\mathscr{X}\times \mathscr{X}\to \mathscr{A}$ be a bounded $\mathscr{A}$-sesqulinear form.
An element $s\in \mathscr X$ is said to be a \emph{$B$-spline} if
\[
B(s,y)=0 \qquad 
\]
for all $y\in \mathscr{Y}$.

The \emph{\bf $B$-spline interpolation problem} asks whether for each $x\in \mathscr{X}$ there exists a $B$-spline element $s$ in the coset $x+\mathscr{Y}$.

\vspace{1cm}

\begin{example}
\begin{itemize}
\item[]
\item[(1)] 
Suppose that $P$ is a non-trivial projection on a Hilbert $C^*$-module $\mathscr{X}$ and set $B(x,y):=\langle P(x),y\rangle,\,\, x,y\in\mathscr{X}$. Then there is an orthogonal decomposition $\mathscr{X}={\rm ran}(P)\oplus\ker(P)$. Let $\mathscr{Z}\subseteq\ker(P)$ be a closed submodule and set $\mathscr{Y}:={\rm ran}(P)\oplus \mathscr{Z}\subseteq\mathscr{X}$. Given $x \in \mathscr{X}$ the element $s \in x+\mathscr{Y}$ can be selected as $s=(1-P)(x)$, i.e. the $B$-spline interpolation has a solution. It might be not unique when $\mathscr{Z}\neq\{0\}$. Indeed, let $z\in\mathscr{Z}$, and let $s=(1-P)(x)+z$. Then $s-x=-P(x)+z\in\mathscr{Y}$ and $B(s,y)=\langle P((1-P)(x)+z),y\rangle=\langle P(z),y\rangle=0$ for any $y\in\mathscr{X}$.

\item[(2)] Consider a Hilbert space $\mathscr{H}$ as a Hilbert $C^*$-module over the $C^*$-algebra  $\mathbb{B}(\mathscr{H})$ of all bounded linear operators on $\mathscr{H}$ under the $C^*$-inner product $[x,y]:=x\otimes y$, where $x\otimes y$ is defined by $(x\otimes y)(z)=\langle z,y\rangle x$, and the actions the usual addition, the scalar multiplication $\lambda\cdot x=\bar{\lambda}x$, and the module right action $x\cdot T=T^*(x)$. Then the $B$-spline interpolation problem has no solution for any given nontrivial closed subspace $\mathscr{Y}$ of $\mathscr{X}$ because of the exhausting set of partial isometries of pairs of one-dimensional subspaces of $\mathscr{H}$. 

\item[(3)] Let $\mathscr{H}$ be an infinite-dimensional Hilbert space, $\mathscr{A}=\mathbb{B}(\mathscr{H})$, and let $\mathbb{K}(\mathscr{H})$ be the norm-closed two-sided ideal of $\mathbb{B}(\mathscr{H})$ of all compact operators on $\mathscr{H}$. Let $\mathscr{X}$ be $\mathscr{A}$ with the Hilbert $\mathscr{A}$-module operations inherited from the algebraic operations in $\mathscr{A}$, in particular,  $\langle T,S \rangle = T^*S$. Let $\mathscr{Y}$ be $\mathbb{K}(\mathscr{H})$. Then for $B(\cdot,\cdot)=\langle \cdot,\cdot \rangle$ and $\mathscr{Y}$, the B-spline interpolation problem has no solution. Similar results hold for any non-unital $C^*$-algebra $\mathscr{Y}=\mathscr{B}$ and its multiplier $C^*$-algebra $\mathscr{A}=\mathscr{X}=M(\mathscr{B})$.
\end{itemize}
\end{example}

In \cite[Theorem 2.3.6]{R} it is shown that in the setting of Hilbert spaces if $\sigma$ is a bounded sesquilinear form on $\mathscr{H}$, then there is a unique bounded linear operator $U$ on $\mathscr{H}$ such that 
\[
\sigma(x,y)=\langle U(x),y\rangle.
\]
In the next theorem, we show that the above representation is valid in a self-dual Hilbert $C^*$-module. To achieve it we need a lemma.
%=================================================%

\begin{lemma}\label{thsman}
	Let $\mathscr{X},\mathscr Z$ be Hilbert $\mathscr{A}$-modules. Let $B:\mathscr{X}\times \mathscr{Z}\to \mathscr{A}$ be a bounded $\mathscr{A}$-sesquilinear form on $\mathscr{X}\times \mathscr{Z}$.
		Then there is a unique bounded $\mathscr{A}$-linear map $T: \mathscr{X} \to \mathscr{Z}'$ such that 
	\[
	B (x,z)=T(x)(z)\qquad (x\in \mathscr{X}, \, z\in\mathscr{Z}).
	\] 
\end{lemma}
	\begin{proof}
	For any $x\in \mathscr{X}$ define $T_x:\mathscr{Z}\to \mathscr{A}$ by $T_x(z)=B(x,z)$. It is easy to verify that $T_x\in \mathscr Z'$ and $T:\mathscr{X}\to \mathscr{Z}'$ defined by $T(x)=T_x\,\, (x\in \mathscr{X})$ is a bounded $\mathscr{A}$-linear map. The uniqueness is obvious.
\end{proof}
%=================================================%
\begin{theorem}\label{ths}
	  Let $\mathscr{X}$ be a self-dual Hilbert $\mathscr{A}$-module. Let $B:\mathscr{X}\times \mathscr{X}\to \mathscr{A}$ be a bounded $\mathscr{A}$-sesqulinear form on $\mathscr{A}$.
		Then there is a unique operator $T\in \mathscr L(\mathscr{X})$ such that 
	\[
	B (x,y)=\langle T(x),y\rangle\qquad (x,y\in \mathscr{X}).
	\] 
\end{theorem}
	\begin{proof}
Following the construction in Lemma \ref{thsman} and due to $\mathscr{X}$ is self-dual,  there is $w_x\in \mathscr{X}$ such that $T_x=\widehat{w_x}$. Let $T:\mathscr{X}\to \mathscr{X}$ be defined by $$T(x)=w_x\qquad (x\in \mathscr{X}).$$
One can easily verify that $T$ is a bounded $\mathscr{A}$-linear map. Since $\mathscr{X}$ is self-dual, by \cite[Proposition 2.5.2]{MT}, $T$ is adjointable.
\end{proof}
%=================================================%

Let $B:\mathscr{X}\times \mathscr{X}\to \mathscr{A}$ be a bounded $\mathscr{A}$-sesqulinear form on a Hilbert $\mathscr{A}$-module $\mathscr{X}$. Let $\mathscr{Y}$ be a closed submodule of $\mathscr{X}$.
Set 
\begin{align*}
\widecheck{\mathscr{Y}}=\{ \widecheck{y}\in \mathscr{Y}:B(\widecheck{y},y)=0\,\, {\rm ~for~all~} y\in \mathscr{Y}\},\\
\widetilde{\mathscr{Y}}=\{\widetilde{y}\in \mathscr{Y}: B(y,\widetilde{y})=0\,\, {\rm ~for~all~} y\in \mathscr{Y}\},
\end{align*}
and 
\[
\mathscr{Y}_1=\{y\in \mathscr{Y}:B(y,y)=0\}.
\]

Clearly, $\widecheck{\mathscr{Y}}\subseteq \mathscr{Y}_1$ and $\widetilde{\mathscr{Y}} \subseteq \mathscr{Y}_1$. For example, if $B(\cdot,\cdot)$ is skew-symmetric, i.e. $B(x,y)=-B(y,x)$ on $\mathscr{X}$, then always $\mathscr{Y}_1 = \mathscr{Y}$, but the other two sets are most often smaller. The sets $\widecheck{\mathscr{Y}}$ and $\widetilde{\mathscr{Y}}$ are called the right and left radical of $\mathscr{Y}$, respectively.
 Moreover, for bounded $\mathscr{A}$-sesquilinear forms both $\widecheck{\mathscr{Y}}$ and $\widetilde{\mathscr{Y}}$ are norm-complete $\mathscr{A}$-submodules of $\mathscr{Y}$. The more, $\mathscr{Y}_1$ is invariant under the action of $\mathscr{A}$.
%=================================================%
\begin{proposition} \label{uniqueness}
Let $\mathscr{X}$ be a Hilbert $\mathscr{A}$-module, $\mathscr{Y}$ be a closed submodule of $\mathscr{X}$, and $B:\mathscr{X}\times \mathscr{X}\to \mathscr{A}$ be a bounded $\mathscr{A}$-sesqulinear form on $\mathscr{X}$. Suppose that the $B$-spline interpolation problem has a solution for $\mathscr{Y}$ for an element $x\in {\mathscr{X}}$. Then the following two conditions are equivalent:
\begin{enumerate}
\item The solution of the B-spline problem for $x$ is unique.
\item $\widecheck{\mathscr Y}=\{0\}$.
\end{enumerate}
Since condition (ii) is a global one, any other existing solution of the $B$-spline problem with respect to $\mathscr{Y}$ for other elements $x \in \mathscr{X}$ has to be always unique in the case when condition (ii) holds.
\end{proposition}

\begin{proof}
Suppose, $\widecheck{\mathscr Y}=\{0\}$. Take an element $x \in  {\mathscr{X}}$ and two derived solutions of the B-spline problem for x, say $s_1$ and $s_2$. Then $B(s_1-s_2,y)=B(s_1,y)-B(s_2,y)=0$ for any $y \in {\mathscr{Y}}$ by definition. If $s_1=x+y_1$ and $s_2=x+y_2$ with $y_1, y_2 \in \mathscr{Y}$, then this equality is transformed to the equality $B(y_1-y_2,y)=0$ for any $y \in {\mathscr{Y}}$, i.e. $y_1=y_2$ by our assumption, and hence, $s_1=s_2$. So the solution is unique.

Conversely, if $\widecheck{\mathscr Y}\not=\{0\}$, then there is an element $y_0 \in \widecheck{\mathscr Y} \subseteq \mathscr{Y}$ with $y_0 \not= 0$. For a given element $x \in {\mathscr{X}}$ and for a given solution $s$ of the B-spline problem for $x$ the element $s+y_0$ is also a solution of the B-spline problem for $x$, and both these solutions are different. 
\end{proof}
%=================================================%

\begin{lemma}\label{lc}
	Let $\mathscr{X}$ be a Hilbert $\mathscr{A}$-module and $\mathscr{Y}$ be a closed submodule of $\mathscr{X}$.
	Let $B:\mathscr{X}\times \mathscr{X}\to \mathscr{A}$ be a bounded $\mathscr{A}$-sesquilinear form and 
        let $B$ be positive on $\mathscr{Y}$, i.e. $B(y,y)\geq 0$ for any $y\in \mathscr{Y}$. Then
	\[
	\widecheck{\mathscr{Y}}=\mathscr{Y}_1=\widetilde{\mathscr{Y}}.
	\]
\end{lemma}
\begin{proof}
	Let $f$ be any positive linear functional on $A$. Then $f(B(\cdot,\cdot))$ (simply denoted by $fB(\cdot,\cdot)$~) is a positive $\mathbb{C}$-sesquilinear form on $\mathscr{Y}$. It follows from the Cauchy--Schwartz inequality (see \cite[page 52]{R}) that
	\[
	|fB(x,y)|\leq (fB(x,x))^{\frac{1}{2}}(fB(y,y))^\frac{1}{2}\,,\qquad x,y\in \mathscr{Y},
	\]
and so $\mathscr{Y} _1\subset \widecheck{\mathscr{Y} }$ and $\mathscr{Y} _1\subset\widetilde{\mathscr{Y}}$. 
\end{proof}
%===============================================%
\begin{proposition}\label{lo2}
      Let $\mathscr X$ be a Hilbert $\mathscr{A}$-module and $T: \mathscr{X}\to \mathscr{X}$ be adjointable. Let $B_1:\mathscr{X}\times \mathscr{X}\to \mathscr{A}$ 
     be the bounded $\mathscr{A}$-sesquilinear form defined by $B_1(x,y)=\langle T(x),y \rangle$ for $x,y \in \mathscr{X}$. Let $B_2:\mathscr{X}\times \mathscr{X}\to \mathscr{A}$ be the bounded $\mathscr{A}$-sesquilinear form defined by $B_2(x,y)=\langle T^*(x),y \rangle$ for $x,y \in \mathscr{X}$.
      Then
      $\widetilde{\mathscr{X}_{B_1}}=\widecheck{\mathscr{X}_{B_2}}$ and $\widetilde{\mathscr{X}_{B_2}}=\widecheck{\mathscr{X}_{B_1}}$.
\end{proposition}
\begin{proof}
\begin{eqnarray*}
     \widetilde{\mathscr{X}_{B_1}}  & = &  \{ \tilde{x} : B_1(x,\tilde{x})=0 \,\, {\rm for} \,\, {\rm any} \,\, x \in \mathscr{X} \} \\
                                                   & =  &  \{ \tilde{x} : \langle T(x), \tilde{x} \rangle=0 \,\, {\rm for} \,\, {\rm any} \,\, x \in \mathscr{X} \} \\
                                                   & =  &  \{ \tilde{x} : \langle \tilde{x}, T(x) \rangle=0 \,\, {\rm for} \,\, {\rm any} \,\, x \in \mathscr{X} \} \\
                                                   & = &  \{ \tilde{x} : \langle T^*(\tilde{x}), x \rangle=0 \,\, {\rm for} \,\, {\rm any} \,\, x \in \mathscr{X} \} \\
                                                   & =  &  \{ \tilde{x} : B_2(\tilde{x},x)=0 \,\, {\rm for} \,\, {\rm any} \,\, x \in \mathscr{X} \} \\
                                                   & =  &  \widecheck{\mathscr{X}_{B_2}}  \, . 
\end{eqnarray*}
Analogously, one can prove the other set identity.
\end{proof}
%=================================================%

We say that a bounded $\mathscr{A}$-sesquilinear form $B$ is normal on an orthogonally complemented Hilbert $\mathscr{A}$-submodule $\mathscr{Y}$ of 
a Hilbert $\mathscr{A}$-module $\mathscr{X}$ if $B(y,z)=\langle T(y),z \rangle$ for $y,z \in \mathscr{Y}$ with $T\in \mathscr L(\mathscr{X})$, and $PTP$ 
and $PT^*P$ have the same kernel, where $P$ is the projection of $\mathscr{X}$ onto $\mathscr{Y}$.
%=================================================%

\begin{proposition}\label{lo}
	Let $\mathscr{X}$ be a self-dual Hilbert $\mathscr{A}$-module. Let $\mathscr{Y}$ be an orthogonally complemented  
        Hilbert $\mathscr{A}$-submodule of $\mathscr{X}$. Let $B$ be a normal bounded $\mathscr{A}$-sesquilinear form 
        on $\mathscr{Y}$. Then $$\widecheck{\mathscr{Y}}=\widetilde{\mathscr{Y}}.$$
\end{proposition}
\begin{proof}
Suppose $T\in \mathscr{L}(\mathscr{X})$ such that
\[
B(x,y)=\langle Tx,y\rangle\qquad (x,y\in \mathscr{X}) \, .
\]
Let $T_\mathscr{Y}=P_{\mathscr{Y}}TP_{\mathscr{Y}}   \in \mathscr{L}(\mathscr{Y})$, where $P_{\mathscr{Y}}$ is the projection onto $\mathscr{Y}$. 
Hence,
\begin{eqnarray*}
\widecheck{\mathscr{Y}}
&= & \{ \widecheck{y}\in \mathscr{Y}:B(\widecheck{y},y)=0 \,\, {\rm ~for~all~} y \in \mathscr{Y}\} \\
&= &   \{ \widecheck{y}\in \mathscr{Y}:\langle T_\mathscr{Y}\widecheck{y},y\rangle =0\,\, {\rm ~for~all~} y\in \mathscr{Y}\}\\
&= &   \mathscr{N}(T_\mathscr{Y})\\
&= &   \mathscr{N}(T_\mathscr{Y}^*) \\
&= &   \mathscr{R}(T_\mathscr{Y})^\perp\\
&= &   \{ \widetilde{y}\in \mathscr{Y}:\langle T_\mathscr{Y} y,\widetilde {y}\rangle =0 \,\, {\rm ~for~all~} y\in \mathscr{Y}\}\\
&= &   \{\widetilde{y}\in \mathscr{Y}: B(y,\widetilde{y})=0\,\, {\rm ~for~all~} y\in \mathscr{Y}\}\\
&= &  \widetilde{\mathscr{Y}}.
\end{eqnarray*} 
by \cite[Lemma 15.3.5]{WO} in step 4.
\end{proof}

As a particular case, $PTP$ can be a normal operator. In any case, the operator $T_\mathscr{Y}$ can be extended to a bounded $\mathscr{A}$-linear isomorphism $S$ on $\mathscr{Y}$ setting $S={\rm id}$ on the common kernel and $S=T_{\mathscr{Y}}$ on the orthogonal complement of the kernel with respect to $\mathscr{Y}$. So one sees, any modular isomorphism of that orthogonal complement would suffice and normality of $T_{\mathscr{Y}}$ is not necessary. \\
%----------------------------------------------------------------------------------------%

The following result is a generalization of \cite[Theorem 1.5]{ABR}.

\begin{theorem} \label{uniqueness}
Let $\mathscr{X}$ be a Hilbert $\mathscr{A}$-module. Let $\mathscr{Y}$ be a closed submodule of $\mathscr{X}$. 	Let $B:\mathscr{X}\times \mathscr{X}\to \mathscr{A}$ be a bounded $\mathscr{A}$-sesqulinear form on $\mathscr{X}$ and $T$ be as in Lemma \ref{thsman} and $\widecheck{\mathscr Y}=\{0\}$.  Then  $B$-spline interpolation problem has a solution for $\mathscr{Y}$ if and only if 
$$\{Tx|_\mathscr Y:x\in \mathscr X\}\subseteq \{Ty|_\mathscr Y:y\in \mathscr Y\}$$
\end{theorem}
\begin{proof}
Suppose that $\{Tx|_\mathscr Y:x\in \mathscr X\}\subseteq \{Ty|_\mathscr Y:y\in \mathscr Y\}$ and $z\in\mathscr{X}$. Then there is $y_0\in \mathscr{Y}$ such that $Tz|_\mathscr Y=Ty_0|_\mathscr Y$. Let $s=z-y_0$. Clearly,
\[
B(s,y)=B(z,y)-B(y_0,y)=T(z)(y)-T(y_0)(y)=0
\]
 for all $y\in \mathscr{Y}.$

To prove the converse, let $z\in \mathscr{X}$ be arbitrary. There is $s\in z+\mathscr{Y}$ such that $B(s,y)=0$ for all $y\in \mathscr{Y}$. Since $\widecheck{\mathscr Y}=\{0\}$, there exists a unique element $y_0\in \mathscr{Y}$ such that $s=z+y_0$. Then  
\begin{align*}
(Tz)(y)= B(z,y)=B(-y_0,y)=T(-y_0)(y),
\end{align*}
for all $z\in \mathscr{X}$ and $y\in \mathscr Y$.  Thus $Tz|_\mathscr Y=T(-y_0)|_\mathscr Y$.
\end{proof}

\begin{remark}
The exact version of Theorem 1.5 of \cite{ABR}, stated in the setting of Hilbert spaces, is not valid in the content of (self-dual) Hilbert $C^*$-modules, in general. In fact, if it held, we concluded that ``the $B$-spline interpolation problem has a solution for $\mathscr{Y}$ if and only if $T(\mathscr{X})\subseteq T(\mathscr{Y})$''.\\To see that the above assertion may not hold in the Hilbert $C^*$-module setting, suppose $\mathscr {X}$ be  a self-dual Hilbert $C^*$-module. Let $P$ and $Q$ be two distinct non-trivial projections on $\mathscr{X}$ such that $QP=PQ=P$. Set $T(x)=Q(x)$ for $x \in \mathscr{X}$ and $\mathscr{Y}=P(\mathscr{X})$. Then $T(\mathscr{X})$ is not contained in $\mathscr{Y}$, $T({\mathscr{Y}})=\mathscr{Y}$, $\widecheck{\mathscr Y}=\{0\}$, and for any $x \in \mathscr{X}$ there exists a unique element $s=x-P(x)$ with $P(x) \in \mathscr{Y}$ by definition such that $\langle T(s),y \rangle=0$ for any $y \in \mathscr{Y}$ because
\begin{eqnarray*}
\langle T(s),y \rangle & = & \langle Q(x)-QP(x),y \rangle = \langle Q(x),y>-<P(x),y \rangle \\
                                & = & \langle Q(x),P(y) \rangle - \langle P(x),P(y) \rangle \\
                                & = & \langle P^*Q(x),P(x) \rangle - \langle P(x),P(y) \rangle \\
                                & = & \langle P(x),P(y) \rangle - \langle P(x),P(y) \rangle = 0 
\end{eqnarray*}
for any $y \in \mathscr{Y}$.  But, $T(\mathscr{X})=Q(\mathscr{X})$ is `larger' than $\mathscr{Y} = T(\mathscr{Y}) = P(\mathscr{Y})$.
\end{remark}

The next result reads as follows.

\begin{theorem}\label{th11}
		Let $\mathscr{X}$ be a self-dual Hilbert $\mathscr{A}$-module. Let $\mathscr{Y}$ be an orthogonally complemented submodule of $\mathscr{X}$ and $P$ be the projection onto $\mathscr{Y}$. Let $B:\mathscr{X}\times \mathscr{X}\to \mathscr{A}$ be a bounded $\mathscr{A}$-sesqulinear form and positive on $\mathscr{Y}$. Then a necessary condition for the $B$-spline interpolation problem for $\mathscr{Y}$ to have a solution is 
\begin{equation}\label{e1}
B(x,\widecheck{y})=0\,\, {\rm ~for~all~} x\in \mathscr{X}, \widecheck{y}\in \widecheck {\mathscr{Y}}
\end{equation}
If $PT\mathscr{Y}$ is closed, then \eqref{e1} is also sufficient, where $T$ is as in Theorem \ref{ths}.
\end{theorem}
\begin{proof}
	Let $x\in \mathscr{X}$ be arbitrary and let the $B$-spline interpolation problem have a solution. Then there is $y_0\in\mathscr{Y}$ such that $s=x+y_0$ and $B(s,y)=\langle Ts,y\rangle =0$ for all $y\in \mathscr{Y}$. Hence,
	\begin{align*}
	B(x,\widecheck{y})&=\langle Tx,\widecheck{y}\rangle =\langle T(s-y_0),\widecheck{y}\rangle=-\langle T(y_0),\widecheck{y}\rangle\\
	&=-B(y_0,\widecheck{y})=0 \qquad \qquad \qquad ({\rm ~by~ Lemma~}\ref{lc})
	\end{align*} 
for all $\widecheck{y}\in \widecheck{\mathscr{Y}}$. 

Conversely, let \eqref{e1} is valid and $PT\mathscr{Y} $ is closed. Define $S: \mathscr{X}\to \mathscr{Y}$ by $S= PTP \in \mathcal L(\mathscr{X}, \mathscr{Y})$. Note that $\mathcal R(S)=PT\mathscr{Y}$ is closed and so
\begin{align}\label{eo}
 \mathscr{Y}=\mathcal{N}(S^*)\oplus \mathcal {R}(S)
\end{align}
We claim that $\mathcal N(S^*)=\widecheck{\mathscr{Y}}$. Indeed
\begin{align*}
\mathcal N(S^*)&=\{y\in \mathscr{Y}:S^*y=0 \}\\
&=\{y\in \mathscr{Y}: \langle y',PT^*Py\rangle =0\,\, {\rm ~for~all~} y'\in \mathscr{Y} \}\\
&=\{y\in \mathscr{Y}: \langle Ty',y\rangle =0\,\, {\rm ~for~all~} y'\in \mathscr{Y} \}\\
&=\widetilde{\mathscr{Y}}=\widecheck{\mathscr{Y}}\quad\qquad \qquad ({\rm ~by~ Lemma~} \ref{lo}).
\end{align*} 
Let $x\in \mathscr{X}$ be arbitrary. It follows from \eqref{eo} that $PTx=PTy_0+b_0$ for some $y_0\in \mathscr{Y}$ and $b_0\in \widecheck{\mathscr{Y}}$. Set $s=x-y_0\in x+\mathscr{Y}$, we show that $s$ is a $B$-spline. To this end, given any $y\in \mathscr{Y}$ we have $y=y_1+y_2$ such that $y_1\in \widecheck{\mathscr{Y}}$ and $y_2\in \mathcal \mathcal R(S^*)$. Hence
\begin{align*}
B(s,y)&=\langle Ts,y\rangle =\langle b_0,y_2\rangle = 0 
% &=\langle b_0, S^*z\rangle \qquad\qquad \qquad ({\rm for ~some~} z\in \mathscr{Y})\\
% &=\langle b_0,PT^*Pz\rangle=\langle TPb_0,Pz\rangle=\langle Tb_0,Pz\rangle\\
% &=B(b_0,Pz) =0.
\end{align*}
by  \eqref{e1} and because $b_0 \in \widecheck{\mathscr{Y}}$ and $y_2 \in R(S)$.
\end{proof} 
%=================================================%
The following two propositions give some properties inherited from a module to its second dual. They are interesting on their own right and will be used in what follows.  

\begin{proposition}\label{selfman}
Let $\mathscr{X}$, $\mathscr Z$ be Hilbert $\mathscr{A}$-modules over a $C^*$-algebra $\mathscr{A}$. 
Let $B:\mathscr{X}\times \mathscr{Z}\to \mathscr{A}$ be a bounded $\mathscr{A}$-sesqulinear form. 
Then $B$ is uniquely extended to a bounded $\mathscr{A}$-sesqulinear form on $\mathscr{X}''\times\mathscr Z''$.
\end{proposition}
\begin{proof}
Let $x\in \mathscr{X}$. 
Define $T_x:\mathscr{Z}\to \mathscr{A}$ by $T_x(z)=B(x,z)$ for $z\in \mathscr{Z}$. Then $T_x\in \mathscr{Z}'$ and so we can define 
$T:\mathscr{X}\to \mathscr{Z}'$ by $T(x)=T_x$. It is easy to see that $T$ is a bounded $\mathscr{A}$-linear map. 
Set $T':\mathscr{Z}''\to \mathscr{X}'$, $T'(F)(y)=F(Ty)$, $F\in \mathscr Z''$, $y\in \mathscr X$. 

Apply this once again to get $T'':\mathscr{X}''\to \mathscr{Z}'''$, and recall that $\mathscr Z'''$ is canonically isomorphic 
to $\mathscr Z'$ (\cite{MT}, Corollary 4.1.5). Note that $T''|_{\mathscr X}=T$. In fact
\[T''(\dot{x})(F)=\dot{x}(T'(F))=(T'(F)(x))^*=(F(Tx))^*=\dot{(Tx)}(F)\]
for all $F\in \mathscr{Z}''$. Thus $T''(\dot{x})=\dot{Tx}$, and so by using the identification via the canonical inclusion,i.e. the map $\mathscr X\hookrightarrow\mathscr X'', x\mapsto\dot x$, we get $T''(x)=Tx$.
	
Now, we define 
\[
\widetilde B:\mathscr{X}''\times \mathscr{Z}''\to \mathscr{A}\,,\qquad
	\widetilde B(F,G)= (T''F)(G),\qquad (F\in\mathscr X'',G\in \mathscr{Z}'').
\]

We have
\[\widetilde B(\dot x,\dot z)=T''\dot x(\dot z)=\dot x(T'\dot z)=((T'\dot z)(x))^*=(\dot z(Tx))^*=Tx(z)=B(x,z).\]

 To prove the uniqueness, suppose that $M:\mathscr{X}''\times \mathscr{Z}''\to \mathscr{A}$ is 
an extension of $B$. Then, by Lemma \ref{thsman}, there is $S:\mathscr{X}''\to \mathscr{Z}'$ such that 
$\widetilde B(F,G)-M(F,G)=S(F)(G)$ for all $F,\in \mathscr{X}''$, $G\in\mathscr Z''$. 
As 
$$
S(\dot x)(\dot z)=\widetilde B(\dot x, \dot y)-M(\dot x, \dot z)=0
$$ 
for any $x\in\mathscr X$, $z\in\mathscr Z$, 
we see that $S|_{\mathscr X}=0$. 
Let us show that $S=(S|_{\mathscr X})''$. Note that 
$$
S'(F)(x)=F(Sx)=F(S|_{\mathscr X}(x))=\left(S_{\mathscr{X}}\right)'(F)(x), 
$$
hence $S'=(S|_{\mathscr X})'$. Then 
$$
S''(F)(z)=F(S'(\dot z))=F((S|_\mathscr X)'(\dot z))=(S|_\mathscr X)'(F)(z)
$$ 
for any $F\in\mathscr X''$, $z\in\mathscr Z$, hence $S''=(S|_\mathscr X)''$. On the other hand, $S''=S$. 
Thus $S=(S|_\mathscr X)''=0$. This shows that $M=\widetilde B$.
\end{proof}

%=================================================%
\begin{proposition}\label{orthman}
Let $\mathscr{X}$ be a Hilbert $\mathscr{A}$-module over a $C^*$-algebra $\mathscr{A}$ and $\mathscr{Y}$ be an orthogonally 
complemented submodule of $\mathscr{X}$. Then $\mathscr{Y}''$ is an orthogonally complemented submodule of $\mathscr{X}''$.
\end{proposition}
\begin{proof}
Let $\mathscr X=\mathscr Y\oplus\mathscr Z$, and let $p:\mathscr X\to \mathscr Y$,
$q:\mathscr X\to\mathscr Z$ be the corresponding projections. If $f\in\mathscr X'$, $g\in\mathscr Y'$, $h\in\mathscr Z'$, then
the map $(g,h)\mapsto f$, where $f(x)=g(p(x))+h(q(x))$, has an inverse map $f\mapsto (f|_{\mathscr Y},f|_{\mathscr Z})$, hence 
$\mathscr X'=\mathscr Y'\oplus\mathscr Z'$. Similarly, $\mathscr X''=\mathscr Y''\oplus\mathscr Z''$ algebraically, and it remains only
to check that $\mathscr Y''$ and $\mathscr Z''$ are orthogonal to each other. Let $G\in\mathscr Y''$, $H\in\mathscr Z''$, and let
$G_\mathscr X,H_\mathscr X\in\mathscr X''$ be given by $G_\mathscr X(f)=G(f|_{\mathscr Y})$, $H_\mathscr X(f)=H(f|_{\mathscr Z})$. If $F\mapsto \dot F$ denotes the canonical inclusion $\mathscr X''\to\mathscr X'$, then $\dot H_\mathscr X$ satisfies 
\[\dot H_\mathscr X(x)= H_\mathscr X(\widehat x)=H_\mathscr X(\widehat x|_{\mathscr Z})=H(\widehat{q(x)}),\qquad (x\in\mathscr X)\]
since
\[\widehat{x}|_{\mathscr Z}(z)=\langle x,z\rangle=\langle q(x),z\rangle=\widehat{q(x)}(z).\]
Therefore, 
$\dot H_\mathscr X|_{\mathscr Y}=0$. Hence $\langle G_\mathscr X,H_\mathscr X\rangle=G_\mathscr X(\dot H_\mathscr X)=G(\dot H_\mathscr X|_{\mathscr Y})=0$. 
\end{proof}

%=================================================%

We are ready to state our main result of this section.

	\begin{theorem}
		Let $\mathscr{X}$ be a  Hilbert module over a $C^*$-algebra $\mathscr{A}$ and let $\mathscr{Y}$ be an orthogonally complemented submodule of $\mathscr{X}$. Let $B:\mathscr{X}\times \mathscr{X}\to \mathscr{A}$ be a bounded $\mathscr{A}$-sesqulinear form, positive on $\mathscr{Y}$. Assume 
		there exists $c>0$ and $k>0$ such that
		for every $f\in\mathcal P\mathcal S(\mathscr{A})$ and every $x\in\mathscr{Y}\backslash \widecheck{\mathscr Y}$ there exists $y\in\mathscr{Y}$ with $\|y\|=1$ such that $f(|y|^2)\geq k$ and
		\begin{equation}
		|fB(x,y)|^2\geq c f(|x|^2)f(|y|^2).
		\end{equation}
		Then a necessary condition for the $B$-spline interpolation problem for $\mathscr{Y}$ to have a solution is 
		\begin{equation}\label{e2}
		B(x, \widecheck {y})=0 \qquad (x\in \mathscr{X},\widecheck {y}\in \widecheck {\mathscr{Y}}).
		\end{equation}
		If $\dot{\mathscr{Y}}$ is orthogonally complemented in $\mathscr{X}''$, then \eqref{e2} is also sufficient.
	\end{theorem}
\begin{proof}
Let there be a solution for the $B$-spline interpolation problem. Let $x\in \mathscr{X}$ be arbitrary. Hence, there is an element $s=x+y_0$ with $y_0\in \mathscr{Y}$ such that $B(s, y)=0$ for any $y\in \mathscr{Y}$. Then for any $ \widecheck y\in \widecheck {\mathscr
	Y}$ we have 
\begin{align*}
B(x,\widecheck y )=B(s,\widecheck y)-B(y_0,\widecheck y)=0\qquad({\rm by ~Lemma ~}\ref{lc})
\end{align*}

Conversely, let \eqref{e2} be valid and let $\dot{\mathscr{Y}}$ be orthogonally complemented in $\mathscr{X}''$. Note that if $\mathscr Y=\widecheck{\mathscr Y}$, then \eqref{e2} implies that $B$-spline interpolation problem has a solution. Now suppose the $\mathscr Y\neq \widecheck{\mathscr Y}$. Let $\widetilde B$ and $\widetilde T$ be as in Proposition \ref{selfman}. By assumption, $\mathscr X''$ is self-dual. We shall show that $\widetilde B$ fulfulls the conditions of Theorem \ref{th11}.

1. For any $y\in \mathscr{Y}$ we have 
\[
\widetilde B(\dot y,\dot y)=B(y,y)\geq 0.
\]
2. Denote by $B_0$ the restriction of $B$ onto $\mathscr X\times\widecheck{\mathscr Y}$, and let $(B_0)\widetilde{\phantom{a}}$ denote the extension of $B_0$ to $\mathscr X''\times\widecheck{\mathscr Y}''$ as in Proposition \ref{selfman}. Then, by uniqueness, $\widetilde B|_{\mathscr X\times\widecheck{\mathscr Y}}=(B_0)\widetilde{\phantom{a}}$. If $B_0=0$, then $(B_0)\widetilde{\phantom{a}}=0$, hence $\widetilde B(F,G)=0$ for any $F\in\mathscr X''$, $G\in\widecheck{\mathscr Y}''$.  

3. Let $P$ be the projection of $\mathscr{X}''$ onto $\dot{\mathscr{Y}}$. Now we show that $P T''\dot{\mathscr{Y}}$ is closed in $\mathscr{X}''$. Let $x\in \mathscr{Y}\backslash\widecheck {\mathscr Y}$. Then there is $f_0\in \mathcal P\mathcal S(\mathscr{A})$ such that $f_0(|x|^2)=\|x\|^2$. By assumption, there exists an element $y'\in \mathscr{Y}$ such that $\|y'\|=1$ and $f_0(|y'|^2)\geq k$ and 
	\[
	|f_0B(x,y')|^2\geq cf_0(|x|^2)f_0(|y'|^2).
	\]
	We have 
	\begin{align}\label{msm3}
	\|P T'' \widehat x\|&=\sup_{f\in \mathcal P\mathcal S(\mathscr{A}),\|y\|=1}|f(\langle P  T'' \dot x,\dot y\rangle)| \nonumber\\
	&=\sup_{f\in \mathcal P\mathcal S(\mathscr{A}),\|y\|=1}|f\widetilde B(\dot x,\dot y)|\nonumber\\
	&=\sup_{f\in \mathcal P\mathcal S(\mathscr{A}),\|y\|=1}|fB(x,y)|\nonumber\\
	&\geq |f_0B(x,y')|\nonumber\\
	&\geq \left(cf_0(|x|^2)f_0(|y'|^2)\right)^\frac{1}{2}\nonumber\\
	&\geq c^\frac{1}{2}\,k^\frac{1}{2}\|x\|
	\end{align}
In addition, if $x\in \widecheck{\mathscr Y}$, then we have 
	\begin{align*}
	\|P T''\dot x\|&=\sup_{f\in \mathcal P\mathcal S(\mathscr{A}),\|y\|=1}|f(\langle P ''T \dot x,\dot y\rangle)| \nonumber\\
	&=\sup_{f\in \mathcal P\mathcal S(\mathscr{A}),\|y\|=1}|f\widetilde B(\dot x,\dot y)|\nonumber\\
	&=\sup_{f\in \mathcal P\mathcal S(\mathscr{A}),\|y\|=1}|fB(x,y)|\nonumber=0 \quad ({\rm by~the~ definition~of~} \widecheck {\mathscr{Y}})
	\end{align*}
	Suppose $P T''\dot y_n\to F$ in $\mathscr{X}''$. Two cases occur:\\
(i) There is an infinite number of $y_{n}\in \widecheck{\mathscr Y}$. Then there is a subsequence $\{y_{n_k}\}$ of $\{y_n\}$ such that $y_{n_k}\in \widecheck{\mathscr Y}$. Hence $\|P T''\dot y_{n_k}\|=0$ and so $F=0$. Thus $F\in P T'' \dot{\mathscr Y}$. \\
(ii) There is a finite number of $y_{n}\in \widecheck{\mathscr Y}$. Then, by removing them, we can assume that there is no $y_{n_k}\in \widecheck{\mathscr Y}$. Then, from \eqref{msm3}, we infer that the sequence $\{y_n\}$ converges to $y'\in \mathscr{Y}$, since $\dot{\mathscr{Y}}$ is closed in $\mathscr{X}''$. Hence, $P T''\dot y_n\to P T''y'$. Therefore, $F=P T''\dot y'\in \dot{\mathscr{Y}}$. Thus, $P T'' \dot{\mathscr{Y}} $ is closed. 
\end{proof}
%=================================================%

\section{$B$-splines in Hilbert $W^*$-modules}

We need some auxiliary results. The first two lemmas can be deduced from Propositions \ref{selfman} and \ref{orthman} by noting that if $\mathscr{X}$ is a Hilbert $C^*$-module over a $W^*$-algebra, then $\mathscr{X}'$ is self-dual, and so $\mathscr{X}''=\mathscr{X}'$. 

\begin{lemma}\label{self}
Let $\mathscr{X}$ be a Hilbert $\mathscr{A}$-module over a $W^*$-algebra $\mathscr{A}$. 
Let $B:\mathscr{X}\times \mathscr{X}\to \mathscr{A}$ be a bounded $\mathscr{A}$-sesqulinear form on $\mathscr{X}$. 
Then $B$ is uniquely extended to a bounded $\mathscr{A}$-sesqulinear form on $\mathscr{X}'$.
\end{lemma}

\begin{lemma}\label{orth}
Let $\mathscr{X}$ be a Hilbert $\mathscr{A}$-module over a $W^*$-algebra $\mathscr{A}$ and $\mathscr{Y}$ be an orthogonally complemented submodule of $\mathscr{X}$. Then $\mathscr{Y}'$ is an orthogonally complemented submodule of $\mathscr{X}'$.
\end{lemma}

%=================================================%
\begin{lemma}\label{lweak}
Let $\mathscr{X}$ be a Hilbert $\mathscr{A}$-module over a $W^*$-algebra $\mathscr{A}$. Let $f\in \mathcal {PS} $ and $\tau\in \mathscr{X}'$. Suppose $x_n\in \mathscr{X}$ is such that $x_n+\mathscr N_f\to \tau_f$. Then $(S\widehat x_n)_f\to (S\tau)_f$ for any adjointable operator $S:\mathscr{X}'\to\mathscr{X}'.$
\end{lemma}
\begin{proof}
	At first we show that $(S\widehat x_n)_f\to (S\tau)_f$ weakly in the Hilbert space $\mathscr{H}_f$. To show this fact, suppose that $x\in \mathscr{X}$. We have
	\begin{align*}
	\lim_n(x+\mathscr N_f,(S\widehat {x}_n)_f)_f&=\lim_nf((S\widehat x_n)(x))\qquad ({\rm by~} \eqref{Pas2})\\
	&=\lim_n f(\langle S\widehat x_n, \widehat x\rangle) \qquad ({\rm by~} \eqref{Pas3}) \\
	&=\lim_n f(\langle \widehat x_n, S^*\widehat x\rangle)\\
	&=\lim_n\overline{f(\langle S^*\widehat x,\widehat x_n\rangle)}\\
	&=\lim_n((S^*\widehat x)_f, x_n+\mathscr N_f)_f\\
	&=((S^*\widehat x)_f,\tau_f)_f\\
	&=f(\langle S^* \widehat x, \widetilde \tau\rangle)\\
	&=f(\langle \widehat x, S\tau\rangle)\\
	&=(x+\mathscr N_f,(S\tau)_f)_f.
	\end{align*}
	Since $\mathscr{X}/\mathscr N_f$ is dense in $\mathscr{H}_f$, we conclude that $(S\widehat{x}_n)_f$ weakly converges to $(S\tau)_f$.
	
	Secondly, we show that $\|(S\widehat x_n)_f\|_{f}\to \|(S\tau)_f\|_f$. Indeed,
	\begin{align*}
	\lim_n \|(S\widehat x_n)_f\|_f^2&=\lim_n ((S\widehat x_n)_f,(S\widehat x_n)_f)_f\\
	&=\lim_n f(\langle S\widehat x_n,S\widehat x_n\rangle )\\
	&=\lim_n f(\langle x_n,S^*S\widehat x_n\rangle )\\
	&=\lim_n (x+\mathscr N_f,(S^*S\widehat x_n)_f)_f\\
	&=(\tau_f,(S^*S\tau)_f)\\
	&=f(\langle \tau, S^*S\tau\rangle)\\
	&=f(\langle S\tau, S\tau)\\
	&=\|(S\tau)_f\|_f^2.
	\end{align*}
	Now, the assertion is deduced by utilizing \cite[Exercises 4.21 ,p.~80]{Wei}.
\end{proof}
%=================================================%
\begin{lemma}\label{lemmpositive}
	Let $\mathscr{X}$ be a Hilbert $\mathscr{A}$-module over a $W^*$-algebra $\mathscr{A}$ and $\mathscr Y$ be an orthogonally complemented submodule of $\mathscr X$. 
	Let $B:\mathscr{X}\times \mathscr{X}\to \mathscr{A}$ be a bounded $\mathscr{A}$-sesqulinear form on $\mathscr{X}$. Let $\widetilde B$ be the extension $B$ on $\mathscr X'$. If $B$ is positive on $\mathscr Y$ then $\widetilde B$ is positive on $\mathscr Y'$
	\end{lemma}
\begin{proof}
	 Suppose that $\tau\in \mathscr{Y}'$ and $f\in \mathcal {PS}$ is arbitrary. Then there is a sequence $\{y_n\}$ such that $	y_n+\mathcal{N}_f\to \tau_f$ in $\mathscr{H}_f$ from	which we get
	\begin{align*}
	f\widetilde B(\tau,\tau)&=f(\langle \widetilde T\tau,\tau\rangle)\\
	&=((\widetilde T\tau)_f,\tau_f)_f\\
	&=\lim_n ((\widetilde T\widehat y_n)_f,y_n+\mathcal N_f)_f\\
	&=\lim f(\langle \widetilde T \widehat y_n,\widehat{y}_n\rangle)\\
	&=\lim fB(y_n,y_n) \geq 0,
	\end{align*}
where $\widetilde T$ is as in Lemma \ref{self} and Theorem \ref{ths}.
	\end{proof}
The main result of this paper reads as follows.
%=================================================%
\begin{theorem}\label{Ta}
	Let $\mathscr{X}$ be a Hilbert $\mathscr{A}$-module over a $W^*$-algebra $\mathscr{A}$ and $\mathscr{Y}$ be a nontrivial orthogonally complemented submodule of $\mathscr{X}$. Let $B:\mathscr{X}\times \mathscr{X}\to \mathscr{A}$ be a bounded $\mathscr{A}$-sesqulinear form on $\mathscr{X}$ and positive on $\mathscr{Y}$. Let $\widetilde B$ be the extension of $B$ on $\mathscr{X}'$. Assume
	there exist $c>0$ and $k>0$ such that
	for every $f\in\mathcal P\mathcal S(\mathscr{A})$ and for every $x\in\mathscr{Y}\backslash\widecheck{\mathscr{Y}}$ there exists a unit vector $y\in\mathscr{Y}$ such that $f(|y|^2)\geq k$ and
	\begin{equation}\label{main2}
	|fB(x,y)|^2\geq cf(|x|^2)f(|y|^2).
	\end{equation}
	Then, the $\widetilde B$-spline interpolation problem has a solution for $\mathscr{Y}'$ if and only if
	\begin{equation}\label{e21}
	B(x,\widecheck{y})=0\qquad (x\in \mathscr{X}, \widecheck{y}\in \widecheck {\mathscr{Y}}).
	\end{equation} 
\end{theorem}
\begin{proof}
	Let the $\widetilde B$-spine interpolation problem has a solution. It is easy to see that $\widehat{\widecheck{\mathscr{Y}}}\subset \widecheck{\mathscr{Y}'}$. If $x\in \mathscr{X}$, then there is $\eta\in \mathscr{Y}'$ such that $\widetilde B(\widehat x+\eta,\sigma)=0$ for any $\sigma\in \mathscr{Y}'$. Hence for any $\widecheck y\in \widecheck{\mathscr{Y}}$ we have 
	\[
	B(x,\widecheck y)=\widetilde B(\widehat x,\widehat {\widecheck y})=-\widetilde B(\eta, \widehat {\widecheck y})=0.
	\]
	
	For the reverse assertion, let \eqref{e21} be valid and let $\widetilde T$ be as in Lemma \ref{self} and Theorem \ref{ths}. It follows from Lemma \ref{orth} that ${\mathscr{Y}}'$ is orthogonally complemented submodule of $\mathscr{X}'$. Note, that if $\mathscr{Y}=\widecheck{\mathscr{Y}}$, then $\mathscr{Y}'=\widecheck{\mathscr{Y}'}$. Indeed, let $\sigma,\eta \in \mathscr{Y}'$ and $f\in \mathcal P\mathcal S(\mathscr{A})$. Then there are sequences $\{y_n\},\{y'_n\}\subset \mathscr{Y}$ such that $y_n+\mathscr N_f\to \sigma_f$ and $y'_n+\mathscr N_f\to \eta_f$. Utilizing Lemma \ref{lweak}, we have 
	\begin{align*}
		f\widetilde B(\sigma,\eta )&=f \langle P\widetilde T\sigma ,\eta\rangle \\
		&=((P\widetilde T\sigma)_f,\eta_f)_f\\
		&=\lim_n (P\widetilde T\widehat y_n+\mathscr N_f,\widehat {y'_n}+\mathscr N_f)_f\\
		&=\lim_n f\langle P\widetilde T\widehat y_n,\widehat {y'_n}\rangle\\
		&=\lim_n fB(y_n,y'_n)=0
\end{align*}
So $\widetilde B(\sigma,\eta)=0$. Hence $\sigma\in \widecheck{\mathscr{Y}'}$. Thus, by \eqref{e21}, the $\widetilde B$-spline interpolation problem has a solution. Now suppose $\mathscr{Y}\neq \widecheck{\mathscr{Y}'}$.
	 Let $P$ be the projection of $\mathscr{X}'$ onto $\mathscr{Y}'$. We see by Lemma \ref{lemmpositive} that $\widetilde B$ is positive on $\mathscr{Y}'$.
	 
	 We intend to show that $P\widetilde T\mathscr{Y}'$ is closed:
	
	Let $x\in \mathscr{Y}\backslash\widecheck{\mathscr{Y}}$. Then there is $f_0\in \mathcal P\mathcal S(\mathscr{A})$ such that $f_0(|x|^2)=\|x\|^2$. By assumption, there exists a unit element $y_0\in \mathscr{Y}$ such that $f_0(|y_0|^2)\geq k$ and fulfulls \eqref{main}. We have
	\begin{align*}
	\|P\widetilde T \widehat x\|&=\sup_{f\in \mathcal P\mathcal S(\mathscr{A}),\|\sigma\|=1, \sigma\in \mathscr{Y}'}|f(\langle P\widetilde T \widehat x,\sigma\rangle)| \\
	&=\sup_{f\in \mathcal P\mathcal S(\mathscr{A}),\|\sigma\|=1, \sigma\in \mathscr{Y}'}|f(\langle \widetilde T \widehat x,\sigma\rangle)| \\
	&\geq\sup_{f\in \mathcal P\mathcal S(\mathscr{A}),\|y\|=1}|f(\langle \widetilde T\widehat x,\widehat y\rangle)|\\
	&=\sup_{f\in \mathcal P\mathcal S(\mathscr{A}),\|y\|=1}|f\widetilde B(\widehat x,\widehat y)|\\
	&=\sup_{f\in \mathcal P\mathcal S(\mathscr{A}),\|y\|=1}|fB(x,y)|\\
	&\geq |f_0B(x,y_0)|\\
	&\geq \left(cf_0(|x|)f_0(|y_0)|\right)^\frac{1}{2}\\
	&\geq c^\frac{1}{2} k^\frac{1}{2} \|x\|.
	\end{align*}
 
 In addition, if $x\in \widecheck{\mathscr{Y}}$, then for any $f\in \mathcal P\mathcal S(\mathscr{A})$ and $\sigma\in \mathscr{Y}'$ there is $\{y_n\}\subset \mathscr{Y}$ such that $y_n+\mathscr N_f\to \sigma_f$. We have 
\begin{align*}
|f\langle P\widetilde T \widehat x,\sigma\rangle|&=\lim_n |f\langle P\widetilde T\widehat x,\widehat y_n\rangle|=\lim_n |f(\widetilde B(\widehat x,\widehat y_n))|=\lim_n |f(B(x,y_n))|=0
\end{align*}
In virtue of orthogonal complimentary of 
$\mathscr{Y}'$ in $\mathscr{X}'$ (see Lemma \ref{orth}), we get 
	\begin{align*}
	\|P\widetilde T \widehat x\|&=\sup_{f\in \mathcal P\mathcal S(\mathscr{A}), \sigma\in \mathscr{Y}', \|\sigma\|=1}|f\langle P\widetilde T \widehat x,\sigma\rangle|=0.
	\end{align*}
	We show that 
	\begin{equation}\label{e121}
	\|P\widetilde T\tau\|\geq c k\|\tau\| \qquad (\tau \in \mathscr{Y}'\backslash\widecheck {\mathscr{Y}'}).
	\end{equation}
	Let $\tau\in \mathscr{Y}'\backslash \widecheck{\mathscr{Y}'}$. Let $g_0\in \mathcal P\mathcal S(\mathscr{A})$ be such that $\|\tau_{g_0}\|_{g_0}^2=g_0(\langle \tau,\tau\rangle)=\|\tau\|^2$; cf. \eqref{Pas3}. Let $x_n\in \mathscr{Y}$ be such that $x_n+\mathscr N_{g_0}\to \tau_{g_0}$. We can choose always that $x_n\in \mathscr{Y}\backslash \widecheck {\mathscr{Y}}$. In fact, two cases occurs: 

(i) There is an infinite number of $x_n$s that are in $\mathscr{Y}\backslash \widecheck {\mathscr{Y}}$. Then, by passing to a subsequence of $\{x_n\}$ if necessary, we can assume that $x_n\in \mathscr{Y}\backslash \widecheck {\mathscr{Y}}$.

(ii) There is an infinite number of $x_n$ that are in $\widecheck {\mathscr{Y}}$. So there is a subsequence of $\{x_n\}$ in $\widecheck {\mathscr{Y}}$, denoted by the same notation $\{x_n\}$. Because of $\mathscr{Y}\neq \widecheck{\mathscr{Y}}$, we can choose some $0\neq z_n\in \mathscr{Y}\backslash \widecheck{\mathscr{Y}}$ such that $z_n\to 0$. Applying \eqref{Pas0} and \eqref{Pas1}, we get 
	\begin{align*}
	\|z_n+x_n+\mathscr N_{g_0}-\tau_{g_0}\|_{g_0}^2&=(z_n+x_n+\mathscr N_{g_0}-\tau_{g_0},z_n+x_n+\mathscr N_{g_0}-\tau_{g_0})_{g_0}\\
	&=g_0(\langle z_n,z_n \rangle)+(x_n+\mathscr N_{g_0},x_n+\mathscr N_{g_0})_{g_0}+g_0(\langle z_n,x_n \rangle)\\
&\quad +g_0(\langle x_n,z_n \rangle)-g_0(\tau(z_n))-\overline {g_0(\tau(z_n))}\\
		&\quad -(x_n+\mathscr N_{g_0},\tau_{g_0})_{g_0}-(\tau_{g_0},x_n+\mathscr N_{g_0})_{g_0}\\
		&\quad +\|\tau_{g_0}\|_{g_0}^2
	\end{align*}
	This shows that $z_n+x_n+\mathscr N_{g_0}\to \tau_{g_0}$. Also $z_n+x_n\not \in \widecheck{\mathscr{Y}}$.
	
	 Lemma \ref{lweak} gives that $(P\widetilde T\widehat x_n)_{g_0}\to (P\widetilde T\tau)_{g_0}$. On the other hand, by assumption there is a unit vector $y_n\in \mathscr{Y}$ such that $k\leq g_0(|y_n|^2)\leq 1$ and fulfills \eqref{main2}. Then we have 
	\begin{align*}
 ckg_0(\langle x_n+\mathscr N_{g_0},x_n+\mathscr N_{g_0} \rangle)&=	ckg_0(|x_n|^2)\\
 &\leq cg_0(|x_n|^2)g_0(|y_n|^2)\\
	&\leq |g_0(B(x_n,y_n))|^2\\
	&=|g_0(\widetilde B(\widehat x_n,\widehat y_n))|^2\\
	&=|g_0(\langle \widetilde T \widehat{x}_n,\widehat {y}_n \rangle )|^2\\
	&=|g_0(\langle P \widetilde T \widehat{x}_n,\widehat {y}_n \rangle)|^2\\
	&\leq \|(P\widetilde T\widehat x_n)_{g_0}\|_{g_0}^2\|y_n+\mathscr N_{g_0}\|_{g_0}^2\\
	&=\|(P\widetilde T\widehat x_n)_{g_0}\|_{g_0}^2g_0(|y_n|^2)	\\
	&\leq \|(P\widetilde T\widehat x_n)_{g_0}\|_{g_0}^2.
	\end{align*}
Taking limits in the last inequality as $n\to \infty$ we arrive at 
 	\[
	ck\|\tau\|^2=ck\,g_0(\langle \tau,\tau\rangle)\leq \|(P\widetilde T\tau)_{g_0}\|_{g_0}^2={g_0}(\langle P\widetilde T\tau,P\widetilde T\tau\rangle)\leq \|P\widetilde T\tau\|^2
	\]
		Let $P\widetilde T\tau_n\to \tau$. From \eqref{e121} we deduce that $\tau_n$ converges to some $\tau_0\in \mathscr{Y}'$ and so $P\widetilde T\tau_n\to P\widetilde T \tau_0$. Thus, $ \tau=P\widetilde T \tau_0\in P\widetilde T\mathscr{Y}'$.
	
	Next, we shall show that
	\begin{equation}\label{ch1}
	\widetilde{B}(\tau,\sigma)=0 \qquad (\tau\in \mathscr{X}',\sigma\in \widecheck{\mathscr{Y}'}).
	\end{equation}
	To reach this aim, we need some new constructions as follows.
	
	Suppose that $\tau\in \mathscr{X}'$ and $\sigma \in\widecheck{\mathscr{Y}'}$.	
	Note that we can define an $\mathscr{A}$-valued inner product on $\mathscr{Y}/\widecheck {\mathscr{Y}}$ by
	\[
	\langle y_1+\widecheck {\mathscr{Y}},y_2+\widecheck {\mathscr{Y}}\rangle _B=B(y_1,y_2) \qquad (y_1,y_2\in\mathscr{Y}).
	\]
	This inner product is well defined, since $\widecheck {\mathscr{Y}}=\mathscr{Y}_1$ by Lemma \ref{lc}. Similarly, $\mathscr{Y}'/\widecheck {\mathscr{Y}'}$ is a pre-Hilbert $C^*$-module equipped with the inner product 
	\[\langle \sigma_1+\widecheck {\mathscr{Y}'},\sigma_2+\widecheck {\mathscr{Y}'}\rangle_{\widetilde B}=\widetilde B(\sigma_1,\sigma_2)\qquad (\sigma_1, \sigma_2\in \mathscr{Y}').
	\]
	Let $\mathscr{Y}'_{\widetilde B}$ is the completion of $\mathscr{Y}'/\widecheck {\mathscr{Y}'}$. Let $\sigma \in \mathscr{Y}'$, define $\Phi_\sigma \in \left(\mathscr{Y}/\widecheck {\mathscr{Y}}\right)'$ by 
	\[\Phi_\sigma (y+\widecheck{\mathscr{Y}})=\widetilde B(\sigma,\widehat y )\qquad (y\in \mathscr{Y}).
	\]
	Since $\widetilde B(\sigma, \widecheck y)=0$ for every $\widecheck y\in \widecheck{\mathscr{Y}}$, so $\Phi_\sigma$ is well defined. Define $\Phi :\mathscr{Y}'/\widecheck {\mathscr{Y}'}\to \left(\mathscr{Y}/\widecheck {\mathscr{Y}}\right)'$ by
	\begin{align*}
	\Phi(\sigma +\widecheck{\mathscr{Y}'})=\Phi_\sigma \qquad (\sigma\in \mathscr{Y}'). 
	\end{align*}
	It is obvious that $\Phi$ is well defined and $\Phi$ is bounded. Indeed, for any $\sigma\in \mathscr{Y}'$ we have 
	\begin{align*}
	\|\Phi(\sigma +\widecheck {\mathscr{Y}'})\|^2&=\sup_{\|y+\widecheck{\mathscr{Y}}\|_B=1}\|\Phi_\sigma (y+\widecheck{\mathscr{Y}})\|^2\\
	&=\sup_{\|y+\widecheck{\mathscr{Y}}\|_B=1}\|\widetilde B(\sigma,\widehat y )\|^2\\
	&=\sup_{\|y+\widecheck{\mathscr{Y}}\|_B=1}\|\langle \sigma +\widecheck {\mathscr{Y}'},\widehat y+\widecheck {\mathscr{Y}'}\rangle _{\widetilde B}\|^2\\
	&\leq \sup_{\|y+\widecheck{\mathscr{Y}}\|_B=1}\| \widehat y+\widecheck {\mathscr{Y}'}\|_{\widetilde B}^2\| \sigma+\widecheck {\mathscr{Y}'}\|_{\widetilde B}^2\\
	&= \|\sigma + \widecheck {\mathscr{Y}'}\|_{\widetilde B}^2.
	\end{align*} 
	In addition, $\Phi $ is bounded below. To see this, let $\sigma \in \mathscr{Y}'\backslash\widecheck{\mathscr{Y}'}$. Employing \eqref{e121}, we have
	\begin{align*}
	\|\Phi(\sigma +\widecheck {\mathscr{Y}'})\|&=\sup_{\|y+\widecheck{\mathscr{Y}}\|_B\leq1}\|\Phi(\sigma +\widecheck {\mathscr{Y}'})(y+\widecheck{\mathscr{Y}})\|\\
	&=\sup_{\|y+\widecheck{\mathscr{Y}}\|_B\leq1}\|\widetilde B(\sigma,\widehat y)\|\\ &=\sup_{\|y+\widecheck{\mathscr{Y}}\|_B\leq1}\|\langle P\widetilde T\sigma,\widehat y\rangle\|\\
	&\geq \sup_{\|y\|\leq\|\widetilde B\|^{\frac{-1}{2}}}\|\langle P\widetilde T\sigma,\widehat y\rangle\|\\
	&\qquad\qquad\qquad ({\rm since~} \|B(y,y)=\|\widetilde B(\widehat y,\widehat y)\|\leq \|\widehat B\|\,\|y\|^2)\\
	&=\|\widetilde B\|^{\frac{-1}{2}}\|P\widetilde T\sigma \|\\
	&\geq ck\|\widetilde B\|^{\frac{-1}{2}}\|\sigma\|\qquad({\rm by~} \eqref{e121})\\
	&\geq ck\|\widetilde B\|^{-1}\|\sigma +\widecheck{\mathscr{Y}'}\|_{\widetilde B}.
	\end{align*}
	We can extend $\Phi$ on $\mathscr {Y'}_{\widetilde B}$, denoted by the same notation $\Phi$. Thus the range of $\Phi$ is closed. We show that the range $\mathcal{R}(\Phi)$ of this extension is $(\mathscr{Y}/\widecheck{\mathscr{Y}})'$:
	
	Let $\tau\in \mathcal{R}(\Phi)^\perp$. Then
	\[
	\langle \Phi(\sigma+\widecheck {\mathscr{Y}'}),\tau\rangle=0\qquad (\sigma \in \mathscr{Y}').
	\]
	Let $f\in \mathcal P\mathcal S(\mathscr{A})$. We have 
	\begin{align}\label{msm5}
	((\Phi(\sigma+\widecheck {\mathscr{Y}'}))_f,\tau_f)_f =f(\langle \Phi(\sigma+\widecheck {\mathscr{Y}'},\tau)\rangle)=0
	\end{align}
	Since $\tau_f\in (\mathscr{Y}/\widecheck{\mathscr{Y}})_f$ there are $y_n+\widecheck {\mathscr{Y}}+\mathscr N_f$ such that $y_n+\widecheck {\mathscr{Y}}+\mathscr N_f\to \tau_f$ in the norm topology. Let $y\in \mathscr{Y}$ be arbitrary. Then 
	\begin{align*}
	( y+\widecheck{\mathscr{Y}}+\mathscr N_f,y_n+\widecheck {\mathscr{Y}}+\mathscr N_f)_f&=f(\langle y+\widecheck{\mathscr{Y}}, y_n+\widecheck{\mathscr{Y}} \rangle_B)\\
	&=fB(y,y_n)\\
	&=f\widetilde B(\widehat y,\widehat y_n)\\
	&=f\langle\Phi(\widehat y+\widecheck{\mathscr{Y}'}),\widehat{ y_n+\widecheck {\mathscr{Y}}}\rangle \\
	&=((\Phi(\widehat y+\widecheck {\mathscr{Y}'}))_f,y_n+\widecheck{\mathscr{Y}}+\mathscr N_f)_f\to 0 \quad ({\rm by~} \eqref{msm5})
	\end{align*}
	as $n\to \infty$. Therefore, $y_n+\widecheck {\mathscr{Y}}+\mathscr N_f\to0$ weakly. Hence, $\tau_f=0$. Thus $\tau=0$. This shows that $\mathcal{R}(\Phi)^\perp=0$, by \cite[Theorem 3.2]{LAN}.
	
	Next, we prove that $\widetilde B(\tau,\eta)=0$ for every $\tau\in \mathscr{X}',\eta \in \widecheck{\mathscr{Y}'}$. Let $\tau\in \mathscr{X}'$. Set 
	\[
	\psi_\tau:\mathscr{Y}/\widecheck{\mathscr{Y}}\to \mathscr{A}\,,\quad \psi_\tau(y+\widecheck {\mathscr{Y}})=\widetilde B (\tau,\widehat y).
	\]
	We now show that $\psi_\tau$ is well defined:\\
	Let $\tau \in \mathscr{X}'$ and $\widecheck y\in \widecheck{\mathscr{Y}}$. Let $f$ be any normal positive linear functional on $\mathscr{A}$. There is a sequence $\{x_n\}$ in $\mathscr{X}$ such that $x_n+\mathscr N_f\to \tau_f$. We have 
	\begin{align*}
	f\widetilde B(\tau, \widehat{\widecheck y})&=f(\langle \widetilde T\tau, \widehat{\widecheck y}\rangle)\\
	&=f(\langle \tau, \widetilde T^* \widehat{\widecheck y}\rangle)\\
	&=(\tau_f,(\widetilde T^* \widehat{\widecheck y}y)_f)_f\\
	&=\lim_n (x_n+\mathscr N_f,(\widetilde T^* \widehat{\widecheck y})_f)_f\\
	&=\lim_n f(\langle x_n, \widetilde T ^* \widehat{\widecheck y}\rangle) \\
	&=\lim_n f(\langle \widetilde Tx_n, \widehat{\widecheck y} \rangle) \\
	&=\lim_n f \widetilde B(\widehat x_n, \widehat{\widecheck y})\\
	&=\lim_n fB(x_n,\widecheck y)\\
 &=0.
	\end{align*}
	Since $f$ is arbitrary, we reach $\widetilde B(\tau, \widehat{\widecheck y})=0$
	
	It follows from
	\[\widehat{ya}(x)=\langle ya,x\rangle=a^*\langle y,x\rangle=a^*\widehat{y}(x)=(\widehat{y}a)(x)\qquad ({\rm by~} \eqref{msm44})\]
that $\psi_\tau\in (\mathscr{Y}/\widecheck{\mathscr{Y}})'$. Hence there are $\sigma_n\in \mathscr{Y}'$ such that $\lim_n\Phi_{\sigma_n}=\lim_n \Phi(\sigma_n+\widecheck{\mathscr{Y}'})=\psi_\tau$. Let $\epsilon>0$ be given. There exists $n_1$ such that 
	\[
	\| \widetilde B(\tau,\widehat y)-\widetilde B(\sigma_n,\widehat y)\|<\epsilon\qquad (n\geq n_1, y\in \mathscr{Y}).
	\]
	Let $\eta \in \widecheck{\mathscr{Y}'}$. Then for any $f\in \mathcal P\mathcal S(\mathscr{A})$ there are $y_n\in \mathscr{Y}$ such that $y_n+\mathscr N_f\to \eta_f$. Then there is $n_0\geq n_1$ such that $|f\widetilde B(\sigma,\widehat{y_{n_0}})-f\widetilde B(\sigma, \eta)|<\epsilon$, where $\sigma$ is either $\sigma_{n_0}$ or $\tau$. For each $n\geq n_0$, it holds that 
	\begin{align*}
	|f\widetilde B(\tau,\eta)|
	&\leq| f\widetilde B(\tau,\eta)-f\widetilde B(\tau,\widehat{y_{n_0}})|\\
	&\quad +|f\widetilde B(\tau, \widehat{y_{n_0}})-f\widetilde B(\sigma_{n_0},\widehat{y_{n_0}})|\\
	&\quad+ |f\widetilde B(\sigma_{n_0},\widehat{y_{n_0}})-f\widetilde B(\sigma_{n_0},\eta)|+|f\widetilde B(\sigma_{n_0},\eta)|\\ 
	&\leq \epsilon+\epsilon+\epsilon+0=3\epsilon.
	\end{align*}
By Theorem \ref{th11}, the $\widetilde B$-spline interpolation has a solution.
\end{proof}
%=================================================%
Now we give an example in which the conditions of Theorem \ref{Ta} simultaneously occur. 

\begin{example}
	Let $\mathscr B$ be an abelian von Neumann algebra of operators acting on a Hilbert space $\mathscr{H}$. Let $\mathscr A=\mathscr B\oplus \mathscr B$ be the von Neumann algebra of all $u\in \mathbb{B}(\mathscr{H})$ having the representation
\begin{eqnarray}\label{diag2}
	u=\left[\begin{matrix}
	u_1&0\\
	0&u_2
	\end{matrix}\right] \qquad (u_1, u_2\in\mathscr B).
\end{eqnarray}
Let $\mathscr{X}:=\mathscr A$ and $\mathscr{Y}:=\mathscr B\oplus 0$. Define $B:\mathscr{X}\times \mathscr{X}\to\mathscr A$ by
	\[
	B(u,v)=\left[\begin{matrix}
	u_1^*v_1&0\\
	0&0
	\end{matrix}\right],
	\]
where $u$ and $v$ have the representations as presented in \eqref{diag2}. Then $B$ is a bounded $\mathscr{A}$-sesqulinear form on $\mathscr{X}$ and positive on $\mathscr{Y}$. Clearly, $\widecheck{\mathscr{Y}}=\{0\}$. By \cite[Theorem 5.1.6]{R}, any pure state of $\mathscr A$ is multiplicative on $\mathscr A$. Hence,
\[
|fB(u,v)|^2=\left|f\left(\left[\begin{matrix}
u_1^*v_1&0\\
0&0
\end{matrix}\right]\right)\right|^2=\left|f\left(\left[\begin{matrix}
u_1^*&0\\
0&0
\end{matrix}\right]\right)f\left(\left[\begin{matrix}
v_1&0\\
0&0
\end{matrix}\right]\right)\right|^2=f(|u|^2)f(|v|^2)
\]
for all $u, v\in \mathscr{Y}$. It follows from \cite[page 10]{LAN} that $\mathscr{X}$ is self-dual. Hence, Theorem \ref{Ta} ensures that $B$-spline interpolation has a solution for $\mathscr{Y}$. 
 \end{example}

 %=================================================%
In the following technical example we show that \eqref{main} is crucial. 

\begin{example}
Let $\mathscr A =\mathscr B(\mathscr H)$, where $\mathscr{H}$ is a separable Hilbert space with the standard orthonormal basis $(e_i)$.  Let  $l_2( \mathscr A)$ be the set of all sequences $(T_i)$ such that $\sum T_i^*T_i$ converges in the norm topology. The $\mathscr A$-inner product on $\mathscr{X}=l_2(\mathscr A)$ is defined  by
	\[
\langle (T_i),(S_i)\rangle=\sum_{i=1}^\infty T_i^*S_i
	\]
By \cite[Proposition 2.5.5]{MT}, it is known that 
$$l_2(\mathscr A)'=\left\{(T_i):  \left(\sum_{i=1}^NT_i^*T_i\right)_N {\rm ~is ~uniformly ~bounded}\right\}$$
Set 
$$\mathscr{Y}:=\{(T_i)\in l_2(\mathscr A): T_{2i-1}=0 {\rm ~for~} 1=1, 2, \ldots\}.$$ Then $\mathscr{Y}$ is an orthogonally complemented submodule of $ l_2(\mathscr A)$ with
$$\mathscr{Y}^\perp=\{(T_i)\in l_2(\mathscr A): T_{2i}=0 {\rm ~for~} 1=1, 2, \ldots\}.$$
If the map $\phi :l_2(\mathscr A)\to l_2(\mathscr A)$ is defined by 
\[
(\phi(T_i))_j=\begin{cases}
\frac{1}{j}T_j+\frac{1}{\sqrt{j}}T_{j+1}&{\rm ~even~} j\\
0& {\rm ~odd~}j
\end{cases},
\]
then, evidently, $\phi \in \mathscr L(l_2(\mathscr A))$ and
\[(\phi^*(T_i))_j=\begin{cases}
\frac{1}{j}T_j&{\rm ~even~} j\\
\frac{1}{\sqrt{j}}T_j& {\rm ~odd~} j>1\\
0&~j=1
\end{cases}.
\]
Let us define a bounded $\mathscr A $-sequilinear form $B$ on $\mathscr{X}$ by 
	\[
	B((T_i),(S_i))=\langle  \phi(T_i),S_i\rangle=\sum_{j=1}^\infty\left(\frac{1}{2j}T_{2j}+\frac{1}{\sqrt{2j}}T_{2j+1}\right)^*S_{2j}
	\]
Now suppose $\widetilde B$ is the extension of $B$ on $l_2(\mathscr A)'$. It follows from \cite[page 29]{MT} that
\[
\widetilde B((T_i),(S_i))=w-\lim\sum_{j=1}^\infty\left(\frac{1}{2j}T_{2j}+\frac{1}{\sqrt{2j}}T_{2j+1}\right)^*S_{2j} 
\]
 We see that 
 $$\mathscr{Y}'=\{(T_i)\in l_2(\mathscr A)': T_{2i-1}=0 {\rm ~for~} 1=1, 2, \ldots\}.$$
 
 In addition,
$B$ is positive on $\mathscr{Y}$, since for all $(T_i)\in \mathscr{Y}$ we have 
\begin{align*}
B(\phi(T_i),(T_i))&=\sum_{j=1}^\infty \frac{1}{2j}T_{2j}^*T_{2j}\geq 0\,,
\end{align*}
in which we use the fact that 
\[\varphi(0,T_2,0,T_4, \cdots)=\left(0, \frac{1}{2}T_2, 0, \frac{1}{4}T_4, \ldots\right).\]

The $\widetilde B$-spline interpolation has no solution for $\mathscr{Y}'$. To see this, suppose that $(P_i)\in l_2(\mathscr A)'$ be the sequence of projections  $P_i=e_i\otimes e_i$. Note that  $\sum_{i=1}^\infty P_i\to I$ in the strong operator topology and $P_i$s are pairwise orthogonal. For the element $(P_i)$, if there is $(S_i)\in \mathscr{Y}'$ such that $B((P_i)+(S_i),(Q_i))=0$ for all $(Q_i)\in \mathscr{Y}'$, then we have 
\[w-\lim \sum_{j=1}^\infty\left(\frac{1}{2j}(P_{2j}+S_{2j})+\frac{1}{ \sqrt{2j}}P_{2j+1}\right)^*Q_{2j}=0\,.\]
For each $j$, by choosing $(Q_i)\in \mathscr{Y}'$ with $Q_{2i}=I$ for $i=j$ and $Q_{2i}=0$ otherwise, we arrive at
\[
S_{2j}=-P_{2j}-\sqrt{2j}P_{2j+1}.
\]
Then $(S_i)$ cannot be in $\mathscr{Y}'$, since $\left(\sum_{i=1}^NS_i^*S_i\right)_N$ is not uniformly bounded. Indeed,
\begin{align*}\left\|\sum_{j=1}^NS_{2j}^*S_{2j}e_{2j+1}\right\|&=\left\|\sum_{j=1}^N(P_{2j}+\sqrt{2j}P_{2j+1})e_{2j+1}\right\|\\
&=\left\|\sum_{j=1}^N\sqrt{2j}e_{2j+1}\right\|=\sum_{j=1}^N\sqrt{2j}.
\end{align*}
Note that $\widecheck{\mathscr{Y}}=\{0\}$. In fact, suppose $(T_i)\in \widecheck{\mathscr{Y}}$. For each $(S_i)\in \mathscr{Y}$ we have 
\begin{align*}
B((T_i),(S_i))=\langle  \phi(T_i), (S_i)\rangle=\sum_{j=1}^\infty\left(\frac{1}{2j}T_{2j}+\frac{1}{\sqrt{2j}}T_{2j+1}\right)^*S_{2j}=\sum_{j=1}^\infty\frac{1}{2j}T_{2j}^*S_{2j}=0	
\end{align*}
from which by choosing suitable $(S_i)\in \mathscr{Y}$ as above, we get $T_{2i}=0$ for all $i$. Thus $(T_i)=0$.

Moreover,  for any $(T_i)\in  l_2(\mathscr A)$ and $(S_i)\in \widecheck{\mathscr{Y}}$ we have 
	\[
	B((T_i),(S_i))=B((T_i),0))=0
	\] 
Next, we show that \eqref{main} is not valid. If it was true, then there would exist $c>0$ and $k>0$ such that for all $f\in\mathcal P\mathcal S(\mathscr{A})$ and for all $(T_i)\in\mathscr{Y}\backslash\widecheck{\mathscr{Y}}=\mathscr{Y}$ there exists a unit vector $(S_i)\in\mathscr{Y}$ such that $f(|(S_i)|^2)\geq k$ and
	\begin{equation}\label{mainmain}
	|fB((T_i),(S_i))|^2\geq ckf(|( T_i)|^2).
	\end{equation}

Let $j\in\mathbb{N}$ be arbitrary. Consider $(T_i^j)_i=(0,0,\cdots,P_{2j},0,\cdots)\in\mathscr{Y}$. It is known that the linear functional $f_{j}:\mathscr A\to \mathbb C$ by $f_{j} (X)=\langle X(e_{2j}),e_{2j}\rangle $ is a pure state; see \cite[Theorem 5.1.7]{R}.
	 If there is a unit vector $(S_i^j)\in\mathscr{Y}$ such that $f_{j}(\langle (S_i^j),(S_i^j)\rangle)\geq k$, then we have 
	\begin{align*}
	|f_{j}B((T_i^j),(S_i^j))|^2&=\left|\frac{1}{2j}\langle S_{2j}^je_{2j},e_{2j}\rangle \right |^2\leq \left(\frac{1}{2j}\right)^2\|S_{2j}^j\|^2\\
&\leq \left(\frac{1}{2j}\right)^2\|(S_i^j)\|^2=\left(\frac{1}{2j}\right)^2.
	\end{align*}
	On the other hand,
	\[
	f_{j}\left(\left|(T_i^j)\right|^2\right)=\langle P_{2j}(e_{2j}),e_{2j}\rangle=1
	\]
	Hence \eqref{mainmain}
gives that 
\[
ck=ckf_{j}\left(\left|(T_i^j)\right|^2\right)\leq \left|f_{j}B((T_i^j),(S_i^j))\right|^2\leq \left(\frac{1}{2j}\right)^2,
\]	
which is impossible.
\end{example}
 %=================================================%

\begin{corollary}
	Let $\mathscr X$ be a Hilbert $\mathscr A$-module over a $W^*$-algebra and $\mathscr Y$ be an orthogonally complemented submodule of $\mathscr X$. Let $B:\mathscr X\times \mathscr X\to \mathscr A$ be a bounded  $\mathscr A$-inner product on $\mathscr X$.  Assume
	there exist $c>0$ and $k>0$ such that
	for every $f\in\mathcal P\mathcal S(\mathscr{A})$ and for every $x\in\mathscr{Y}$ there exists a unit vector $y\in\mathscr{Y}$ such that $f(|y|^2)\geq k$ and
	\begin{equation*}
	|fB(x,y)|^2\geq cf(|x|^2)f(|y|^2).
	\end{equation*}
	Then $\mathscr Y'$ is an orthogonally complemented submodule of $\mathscr X'$ with respect to the inner product $\widetilde B$.
\end{corollary}
\begin{proof}
	By use of Lemma \ref{lemmpositive} we see that $\widetilde B$ is positive on $\mathscr X'$. Also if $0\neq \tau\in \mathscr X'$ and $f\in \mathcal P(\mathscr A)$ is such that $\|\tau\|^2=f(\langle \tau,\tau\rangle)=\|\tau_f\|^2$, then  there is a sequence $\{y_n\}$ such that $	y_n+\mathcal{N}_f\to \tau_f$ in $\mathscr{H}_f$. For any $y_n\in \mathscr Y$, let unit vector  $z_n\in \mathscr Y$ be such that $f(|z_n|^2)\geq k$ and
	\begin{equation*}
	|fB(y_n,z_n)|^2\geq cf(|y_n|^2)f(|z_n|^2).
	\end{equation*}
With $\widetilde{T}$ as in Lemma \ref{self} and Theorem \ref{ths} we have
	\begin{align*}
	f\widetilde B(\tau,\tau)&=f(\langle \widetilde T\tau,\tau\rangle)\\
	&=((\widetilde T\tau)_f,\tau_f)_f\\
	&=\lim_n ((\widetilde T\widehat y_n)_f,y_n+\mathcal N_f)_f\\
	&=\lim_n f(\langle \widetilde T \widehat y_n,\widehat{y}_n\rangle)\\
	&=\lim_n fB(y_n,y_n) \\
	&\geq \frac{1}{\|B\|} \lim_n fB(y_n,y_n)fB(z_n,z_n) \\
	&\geq \frac{1}{\|B\|} \lim_n \left|fB(y_n,z_n)\right|^2\qquad ({\rm ~by~ the~Cauchy-Schwarz~inequality})\\
	&\geq \frac{1}{\|B\|}\lim_n \left(cf\left(|y_n|^2\right)f\left(|z_n|^2\right)\right)\\
	&\geq \frac{ck}{\|B\|}\lim_nf\left(\langle y_n,y_n\rangle\right)\\
    &= \frac{ck}{\|B\|}f(\langle \tau,\tau\rangle)>0\\
	&= \frac{ck}{\|B\|}\|\tau\|^2>0,
	\end{align*}
which shows that $\widetilde B(\tau,\tau)>0$. Hence $\widetilde B$ is an inner product on $\mathscr X'$. \\
Since $B$ is an $\mathscr A$-valued inner product, $\widecheck{\mathscr Y}=\{0\}$, and so \eqref{e21} holds. Hence, the hypotheses of Theorem \ref{Ta} are fulfilled. Therefore the $\widetilde B$-spline interpolation problem has a solution for $\mathscr{Y}'$. If we set 
	\[
	S_{\widetilde B}=\{s: s {\rm ~is ~ a~} \widetilde B-{\rm spline}\},
	\]
	then for each $\tau \in \mathscr X'$, there are unique elements $s\in S_{\widetilde B}$ and $\rho\in \mathscr Y'$ such that $\tau=s+\rho$. Thus, $\mathscr X'=S_{\widetilde B}\oplus \mathscr Y'$.
\end{proof}

As a consequence, we show when an orthogonally complemented submodule of a self-dual Hilbert $W^*$-module $\mathscr{X}$ is orthogonally complemented with respect to another $C^*$-inner product on $\mathscr{X}$.

\begin{corollary}
		Let $\mathscr X$ be a self-dual  Hilbert $\mathscr A$-module over a $W^*$-algebra and $\mathscr Y$ be an orthogonally complemented submodule of $\mathscr X$. Let $B:\mathscr X\times \mathscr X\to \mathscr A$ be a bounded inner product on $\mathscr X$.  Assume there exist $c>0$ and $k>0$ such that for every $f\in\mathcal P\mathcal S(\mathscr{A})$ and for every $x\in\mathscr{Y}$ there exists a unit vector $y\in\mathscr{Y}$ such that $f(|y|^2)\geq k$ and
\begin{equation}\label{main2}
	|fB(x,y)|^2\geq cf(|x|^2)f(|y|^2).
\end{equation}
Then $\mathscr Y$ is an orthogonally complemented submodule of $\mathscr X$ with respect to the inner product $B$.
\end{corollary}
%=================================================%

\section{Solutions of the $B$-spline interpolation problem for Hilbert $C^*$-modules over $C^*$-ideals of $W^*$-algebras}

Weakly dense, two-sided $C^*$-ideals $\mathscr{A}$ of $W^*$-algebras $\mathscr{D}$ are of interest because for them the $B$-spline interpolation problem for Hilbert $C^*$-modules over them can be solved to the affirmative inside the $W^*$-algebra. These ideals do not have a unit. Examples include the compact $C^*$-algebras inside particular type ${\rm I}$ $W^*$-algebras, but also $C^*$-ideals of ${\rm II}_\infty$ $W^*$-factors like the norm-closures of the set of elements of finite trace or the set of all elements with source or range projections of finite type. 

In the situation characterized above the multiplier algebra $M(\mathscr{A})$ of the $C^*$-algebra $\mathscr{A}$ coincides with the $W^*$-algebra $\mathscr{D}$. Obviously, $\mathscr{D} \subseteq M(\mathscr{A})$ by the two-sided ideal property of $\mathscr{A}$ in $\mathscr{D}$ and by the weak density supposition on $\mathscr{A}$ in $\mathscr{D}$. Conversely, every isometric faithful $*$-representation of a $C^*$-algebra $\mathscr{A}$ on a Hilbert space induces an isometric $*$-representation of its multiplier algebra $M(\mathscr{A})$ inside the bicommutant of this representation; cf. \cite[Proposition 2.2.11 and Example 2.F]{WO}. So $M(\mathscr{A})=\mathscr{D}$.  Note, that strict convergence with respect to $\mathscr{A}$ transfers to the $*$-representation.

By the Cohen--Hewitt factorization theorem (\cite{FD} and \cite[Proposition 2.31]{RW}) for Banach $C^*$-modules any element $x$ of a Hilbert $C^*$-module can be decomposed as $x=ya$ for a certain element $y \in X$ and a certain element $a \in A$. This decomposition is non-trivial for non-unital $C^*$-algebras. Consequently, given the situation under discussion, any Hilbert $\mathscr{A}$-module is a Hilbert $M(\mathscr{A})$-module, too. Moreover, the construction of the $\mathscr{A}$-(bi)dual and $M(\mathscr{A})$-(bi)dual Banach $C^*$-modules of a given Hilbert $\mathscr{A}$-module $\mathscr{X}$ result in exactly the same (bi)dual Banach $C^*$-modules $\mathscr{X}'$ and $\mathscr{X}''$, respectively. So we can conclude $\mathscr{X}' = \mathscr{X}''$ in our situation, since $\mathscr{X}$ and $\mathscr{X}'$ are Hilbert $W^*$-modules over $M(\mathscr{A})=\mathscr{D}$ and Paschke's lifting of the $M(\mathscr{A})$-valued inner product on $\mathscr{X}$ to $\mathscr{X}'$ applies turning the latter into a self-dual Hilbert $W^*$-module over $M(\mathscr{A})$; cf. \cite{Pa}. 

\begin{theorem}
   Let $\mathscr{A}$ be a $C^*$-algebra that is a weakly dense, two-sided ideal in a $W^*$-algebra.
   Let $\mathscr{X}$ be a Hilbert $\mathscr{A}$-module and $\mathscr{Y}$ be a nontrivial orthogonally complemented submodule of $\mathscr{X}$. 
   Let $B:\mathscr{X}\times \mathscr{X}\to \mathscr{A}$ be a bounded $\mathscr{A}$-sesqulinear form on $\mathscr{X}$ and positive on $\mathscr{Y}$. 
   Denote by $\widetilde B$ the extension of $B$ on $\mathscr{X}'$. Assume there exist $c>0$ and $k>0$ such that for every $f	\in\mathcal P\mathcal S(\mathscr{A})$ and for every $x\in\mathscr{Y}\backslash\widecheck{\mathscr{Y}}$ there exists a unit vector $y\in\mathscr{Y}$ such that $f(|y|^2)\geq k$
   and
   \begin{equation}\label{main}
	|fB(x,y)|^2\geq cf(|x|^2)f(|y|^2).
   \end{equation}
   Then, the $\widetilde B$-spline interpolation problem has a solution for $\mathscr{Y}'$ if and only if
   \begin{equation*}
	B(x,\widecheck{y})=0\qquad (x\in \mathscr{X}, \widecheck{y}\in \widecheck {\mathscr{Y}}).
   \end{equation*} 
\end{theorem}
\begin{proof}
Because of the special properties of the $C^*$-algebra $\mathcal{A}$ as an ideal inside a $W^*$-algebra the latter can be identified with the multiplier $C^*$-algebra $M(\mathscr{A})$. Therefore, $\mathscr{X}$ and $\mathscr{Y}$ are Hilbert $M(\mathscr{A})$-modules as well. The self-dual $M(\mathscr{A})$-dual Hilbert $M(\mathscr{A})$-module $\mathscr{X}'$ can be identified with the multiplier module $M(\mathscr{X})$ of $\mathscr{X}$; see \cite{BG1, BG2}. So
we can apply Theorem \ref{Ta} and the result follows.
\end{proof}

By \cite[Theorem 2.4]{Brown15}, among the $\sigma$-unital and unital $C^*$-algebras $\mathscr{A}$ with $W^*$-algebras $M(\mathscr{A})$ only compact $C^*$-algebras and (unital) $W^*$-algebras have this property. The $C^*$-ideals of ${\rm II}_\infty$-$W^*$-factors like the norm-closures of the set of elements of finite trace or the set of all elements with source or range projections of finite type do not have a strictly positive element; see \cite[Proposition 4.5]{AP}. The non-$\sigma$-unital case is not classified yet. However, the $W^*$-algebra of all bounded linear operators on a non-separable Hilbert space possesses more norm-closed two-sided $C^*$-ideals beside the $C^*$-algebra of compact operators, e.g. all bounded linear operators with separable domain and range. Diving into the chain of non-equal cardinalities of sets seen as dimensions of underlying Hilbert spaces reveals more examples, even without assuming the continuum hypothesis.  

%============================================%
\textbf{Acknowledgement.} The authors would like to sincerely thank the referee for some helpful comments improving the paper. The fourth author (corresponding author) is supported by a grant from Ferdowsi Univeristy of Mashhad (No. 2/52682).

\end{document}